\documentclass{amsart}

\usepackage{graphicx}

\usepackage{a4wide,comment}
\usepackage{amssymb}
\usepackage{color}
\usepackage[all]{xy}

\usepackage[format=plain, 
justification=raggedright,singlelinecheck=false]{caption}

\newtheorem{theorem}{Theorem}[section]    
\newtheorem{lemma}[theorem]{Lemma}          
\newtheorem{proposition}[theorem]{Proposition}  
  
\newtheorem{example-definition}[theorem]{Example-Definition} 

\newtheorem{corollary}[theorem]{Corollary} 

\theoremstyle{definition}
\newtheorem{definition}[theorem]{Definition}
\newtheorem{remark}[theorem]{Remark}  
\newtheorem{question}[theorem]{Question} 

\newtheorem{example}[theorem]{Example}    
\newtheorem*{remark-no-number}{Remark}             
\newtheorem*{definition*}{Definition}     
\newtheorem*{example*}{Example}     
\numberwithin{equation}{section}
\newcommand{\ri}{{\sf right}}
\newcommand{\Int}{{\sf int}}
\newcommand{\dep}{{\sf depth}}

\newcommand{\disj}{{\ \prec_{\sf disj}\ }}
\newcommand{\disjp}{{\ \prec_{\sf disj}^{\partial + P}\ }}

\newcommand{\Z}{\mathbb{Z}}

\newcommand{\Diff}{{\rm Diff}^+}
\newcommand{\MCG}{{\mathcal MCG}}

\newcommand{\A}{\mathcal{A}_{\mathcal B}(S, P)}
\newcommand{\AP}{\mathcal{A}^{\partial+P}_{\mathcal B}(S, P)}
\newcommand{\B}{\mathcal{B}}

\title[Twist left-veering open books]{
Twist left-veering open books, overtwistedness, looseness and virtual looseness
}
\author{Tetsuya Ito}
\address{Department of Mathematics, Kyoto University, Kyoto 606-8502, JAPAN}
\email{tetitoh@math.kyoto-u.ac.jp}
\author{Keiko Kawamuro}
\address{Department of Mathematics,   
The University of Iowa, Iowa City, IA 52242, USA}
\email{keiko-kawamuro@uiowa.edu}
\date{\today} 

\begin{document}

\begin{abstract} 
We introduce twist left-veering mapping classes of punctured surfaces. We prove that a twist left-veering open book supports an overtwisted contact structure and determine when the closed braid coming from the punctures is loose or virtually loose. 
\end{abstract}

\maketitle

\section{Introduction}

A pair $(S,\phi)$ of a compact oriented surface $S$ with boundary and diffeomorphism $\phi\in \Diff(S, \partial S)$ is called an {\em (abstract) open book}.  For an open book $(S,\phi)$ one obtains a contact 3-manifold $(M_{(S,\phi)},\xi_{(S,\phi)})$ \cite{TW}.
Open books play significant role in the study of contact structures thanks to the Giroux correspondence, \cite{Giroux}, see also \cite{Et, EV}:  
There is a one-to-one correspondence between the set of contact 3-manifolds up to isotopy and the set of open books up to stabilization and equivalence.
Here we say that two open books $(S,\phi)$ and $(S',\phi')$ are \emph{equivalent}  if $h\circ \phi=\phi' \circ h$ for some orientation preserving diffeomorphism $h: S\rightarrow S'$ fixing the boundary \cite{Et}.

Let $\mathcal T$ be a transverse link in $(M, \xi) \simeq (M_{(S, \phi)}, \phi_{(S, \phi)})$, where $\simeq$ means contactomorphic. According to Bennequin \cite{Ben} (when $(M,\xi)=(S^3,\xi_{std})$, the standard contact 3-sphere) and Pavelescu \cite{pav1, pav} (for general $(M,\xi)$), $\mathcal T$ is transversely isotopic to some closed $n$-braid $L$ with respect to the open book $(S,\phi)$.
Through the contactomorphism $(M, \xi) \simeq (M_{(S, \phi)}, \xi_{(S, \phi)})$ and braid isotopy, $L$ induces a set $P$ of $n$ points in $S$ near the boundary $\partial S$ and a diffeomorphism $\phi_L \in \Diff(S, P, \partial S)$ called the {\em distinguished monodromy} of $L$.

It is a fundamental problem to understand geometry and topology of $(M, \xi, \mathcal{T})$ in terms of the corresponding mapping classes $[\phi]\in\MCG(S)$ and $[\phi_L]\in \MCG(S, P)$. (In the following the bracket $[\cdot]$ will be omitted for simplicity, and a diffeomorphism and its mapping class will be denoted by the same symbol.)  
In this paper, we are particularly interested in detecting tight/overtwisted-ness of $(M, \xi)$ and non-loose/loose-ness of $\mathcal T$.
Three earlier works done by Goodman \cite{Goodman}, Wand \cite{Wand1}, Honda, Kazez and Mati\'c \cite{HKM} are relevant to our goal:

In \cite{Goodman} Goodman introduced a \emph{sobering arc} for an open book $(S,\phi)$ which is a properly embedded arc $\alpha$ in $S$ such that the intersection of $\alpha$ and its image $\phi(\alpha)$ satisfies certain  numerical conditions. 
He showed that {\em $(S,\phi)$ supports an overtwisted contact structure if and only if $(S, \phi)$ is stably equivalent to an open book $(S', \phi')$ admitting a sobering arc}. Here we say that two open books are \emph{stably equivalent} if they admit  stabilizations that are equivalent. 

In \cite{Wand1} Wand introduced an \emph{overtwisted region} in $(S,\phi)$, which can be seen as a generalization of Goodman's sobering arcs to arc systems.
It is a $2N$-gon formed by an $N$-arc system $\Gamma \subset S$ and its image $\phi(\Gamma)$ that satisfies certain conditions (in Definition \ref{def:OT-region}). Here, an $N$-arc system is a collection of pairwise disjoint $N$ arcs. He showed that {\em $(S,\phi)$ supports an overtwisted contact structure if and only if $\phi$ is inconsistent}; that is, there exist some arc system $\Gamma\subset S$ and \emph{stabilization} $(S', \phi')$ of $(S, \phi)$ such that $\Gamma$ and $\phi'(\Gamma)$ form an overtwisted region. This observation lead him to prove that Legendrian surgery preserves tightness.

There is an alternative refinement of Goodman's sobering arc criterion by Honda, Kazez and Mati\'c \cite{HKM}, which is the {\em non-right-veering} arc criterion. Instead of looking at the whole picture of $\gamma\cap \phi(\gamma)$ they found the importance of focussing on $\gamma\cup \phi(\gamma)$ restricted to a neighborhood of the boundary $\partial S$. 
They proved that {\em an open book is overtwisted if and only if it is stably equivalent to a non-right-veering open book.} 
Among the three criteria, the non-right-veering criterion is the most practically easy to check since it is essentially equivalent to the condition that the fractional Dehn twist coefficient (FDTC) is non-positive, see Propositions 3.1 and 3.2 in \cite{HKM}.

Note that the above three overtwistedness criteria are not at all claiming that every open book $(S,\phi)$ supporting an overtwisted contact structure admits sobering arcs, overtwisted regions or non-right-veering arcs. 

Also, in the three criteria, punctured open books are not considered. 
In \cite{IK-qveer} we studied open books $(S, \phi)$ with a set of punctures $P$ where $\phi$ is a diffeomorphism of $(S, P, \partial S)$, and we will continue the study of punctured open books in this paper. 
Since the punctures in $P$ can be permuted by $\phi$, it gives rise to a closed braid in the manifold $M_{(S, \phi)}$ that can be identified as a transverse link in the contact manifold $(M_{(S, \phi)}, \xi_{(S, \phi)})$. 
As a consequence, the overtwisted-structure detection problem is converted to a loose-link detection problem. 

It turned out that extending the results for non-punctured open books to punctured open books is not routine. 
Even though the notion of right-veering can naturally be extended to  punctured open books, the extended non-right-veering property does not imply looseness of the transverse link. To solve this problem, in \cite{IK-qveer} a notion of non-{\em quasi}-right-veering is introduced for punctured open books. This condition appears to be a `correct' generalization of non-right-veering as it implies looseness of the transverse link.

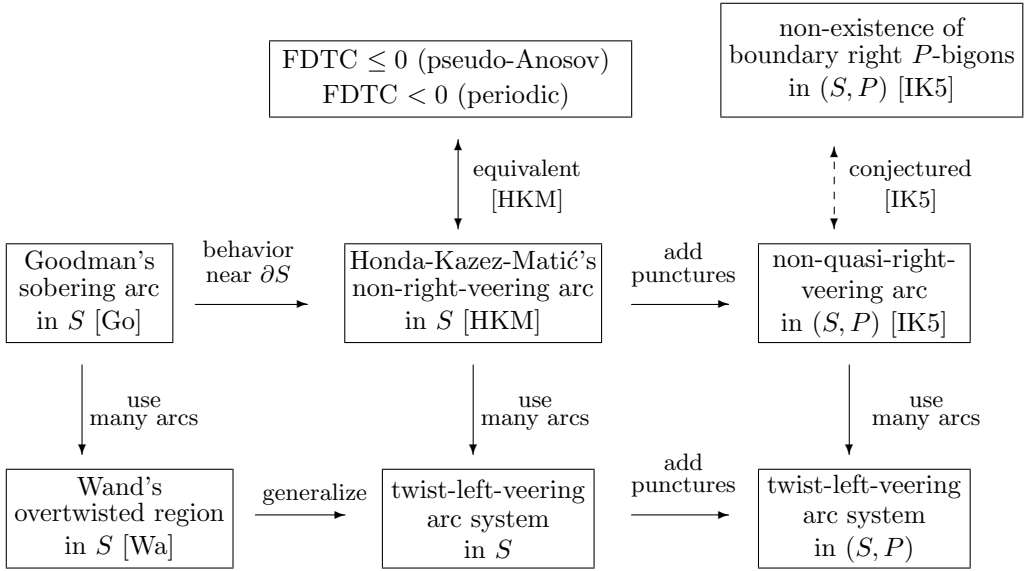
\begin{figure}[h]
\setlength{\unitlength}{1mm}
\begin{picture}(150,75)
\put(0,30){\framebox(22,13){\shortstack{Goodman's\\ sobering arc\\ in $S$ \cite{Goodman}}}}
\put(0,0){\framebox(30,13){\shortstack{Wand's\\overtwisted region\\ in $S$ 
\cite{Wand1}}}}
\put(45,30){\framebox(34,13){\shortstack{Honda-Kazez-Mati\'c's \\non-right-veering arc\\ in $S$
\cite{HKM}}}}
\put(100,30){\framebox(28,13){\shortstack{non-quasi-right- \\veering arc 
\\
in $(S, P)$ \cite{IK-qveer}
}}}
\put(50,0){\framebox(28,13){\shortstack{twist-left-veering\\ arc system \\ in $S$}}}
\put(100,0){\framebox(28,13){\shortstack{twist-left-veering\\ arc system \\in $(S,P)$}}}
\put(35,60){\framebox(48,10){\shortstack{ FDTC $\leq 0$ (pseudo-Anosov)\\ FDTC $<0$ (periodic)} } }
\put(95,60){\framebox(40,15){\shortstack{non-existence of \\
boundary right $P$-bigons \\
in $(S,P)$ \cite{IK-qveer}}}}
\put(60,45){\vector(0,1){12}}
\put(60,45){\vector(0,-1){0}}
\put(62,48){\small  \shortstack{equivalent \\ \cite{HKM}}}
\put(110,56){\vector(0,1){0}}
\put(110,45){\vector(0,-1){0}}
\multiput(110,45)(0,2){6}{\line(0,1){1}}
\put(112,48){\small  \shortstack{conjectured \\ \cite{IK-qveer}}}
\put(25,35){\vector(1,0){15}}
\put(83,35){\vector(1,0){13}}
\put(10,27){\vector(0,-1){12}}
\put(33,7){\vector(1,0){12}}
\put(83,7){\vector(1,0){13}}
\put(62,27){\vector(0,-1){12}}
\put(112,27){\vector(0,-1){12}}
\put(83,38){\small \shortstack{add\\ punctures}}
\put(63,19){\small \shortstack{use\\ many arcs}}
\put(26,38){\small \shortstack{behavior\\ near $\partial S$}}
\put(11,19){\small \shortstack{use\\ many arcs}}
\put(83,10){\small \shortstack{add\\ punctures}}
\put(34,9){\small \shortstack{generalize}}
\put(115,19){\small \shortstack{use\\ many arcs}}
\end{picture}
\caption{Criteria for overtwistedness and looseness.}
\label{fig:table of criteria}
\end{figure}

For non-quasi-right-veering, there is no numerical
characterization like the non-positive FDTC characterization for non-right-veering. 
However, a boundary right $P$-bigon, that is a certain punctured bigon at the boundary $\partial S$, plays a practically useful role to detect quasi-right-veering. In \cite{IK-qveer} it is conjectured that non-existence of boundary P-bigon is equivalent to non-quasi-right-veering.

In this paper, we extend 
the existing arc criteria in \cite{HKM, IK-qveer} to arc-system criteria. 
More precisely, we generalize the notion of non-(quasi)-right-veering to a notion of {\em twist-left-veering}. Our work may be understood as an analogue of Wand's generalization of Goodman's criterion, and also possibly can  be seen as a generalization of Wand's inconsistency criterion (Corollary~\ref{cor:stably-LV}). The schematic picture in Figure~\ref{fig:table of criteria} may be helpful. 

In \cite{IK-qveer} we introduced the \emph{right-veering ordering} $\prec_{\sf right}$ and the \emph{strong right-veering ordering} $\ll_\ri$ of arcs, see Definition~\ref{def:<<single-arc}. A mapping class $\phi \in\MCG(S, P)$ is called {\em non-right-veering} (resp. {\em non-quasi-right-veering}) if $\phi(\gamma)\prec_\ri \gamma$ (resp. $\phi(\gamma)\ll_\ri \gamma$) for some arc $\gamma$. Both orderings $\prec_\ri$ and $\ll_\ri$ can be naturally extended to $N$-arc systems.

Let $N\in \mathbb N$. 
If an $N$-arc system $\Gamma$ and its image $\phi(\Gamma)$ form a $2N$-gon then we call it a {\em boundary based region} and denoted it by $R(\Gamma, \phi(\Gamma))$ (see Definition~\ref{def-of-R} and Figure~\ref{fig:zu-boundary-twist}). 
If a $2N$-gon is formed then we define another $N$-arc system $\phi^{tw}(\Gamma)$ called the {\em left-twist} of $\Gamma$ by $\phi(\Gamma)$ (Definition~\ref{def:left-twist}). We have $\phi^{tw}(\Gamma) \prec_\ri \Gamma \prec_\ri \phi(\Gamma)$ in general. 
If no $2N$-gon is formed we define $\phi^{tw}(\Gamma):=\phi(\Gamma)$.
We say that $\phi$ is {\em $(N,k)$-twist left-veering} if $\phi^{tw}(\Gamma) \ll_\ri \Gamma$ for some $N$-arc system $\Gamma$ such that the boundary based region $R(\Gamma, \phi(\Gamma))$ contains $k$ puncture points. We also say 
that $\phi$ is $N$-twist left-veering (resp. twist left-veering) if $\phi$ is $(N,k)$-twist left-veering for some $k$ (resp. $N$ and $k$).

When $P=\emptyset$, $1$-twist left-veering and non-right-veering are equivalent, and when $P\neq \emptyset$, $(1,0)$-twist left-veering and non-quasi-right-veering are equivalent. 
Thus, twist left-veering is a generalization of non-right-veering and non-quasi-right-veering.  
Although someone might want to name it non-twist right-veering, we prefer twist left-veering to avoid the prefix `non-', and in fact, when $\phi^{tw}(\Gamma)\ll_\ri \Gamma$ the arc system $\phi^{tw}(\Gamma)$ is on the left of $\Gamma$ near the base points.

A contact 3-manifold $(M_{(S,\phi)},\xi_{(S,\phi)})$ is overtwisted if $\phi$ is non-right-veering \cite{HKM}. Similarly, we showed that a transverse link represented by a braid $L$ is loose (the complement is overtwisted) if the distinguished monodromy $\phi_L\in\MCG(S, P)$ is non-quasi-right-veering \cite{IK-qveer}. We generalize these results to $N$-twist left-veering.

\noindent {\bf Theorem~\ref{theorem:non-Nrv-OT}.} 
{\em
Let $L$ be a closed braid with respect to $(S, \phi)$ and $\phi_L \in \MCG(S,P)$ be its distinguished monodromy. 
\begin{enumerate}
\item If $\phi \in \MCG(S)$ is $N$-twist left-veering then $(M_{(S,\phi)},\xi_{(S,\phi)})$ is overtwisted, and there is an overtwisted disk that intersects the binding at $N$ points.
\item 
If $\phi_L \in \MCG(S, P)$ is $(N,k)$-twist left-veering then there is an overtwisted disk that intersects the binding at $N$ points and intersects the closed braid $L$ at $k$ points.
In particular, if $\phi_L$ is $(N,0)$-twist left-veering then $L$ is loose.

\end{enumerate}
}

For a transverse link $\mathcal{T}$ in an overtwisted contact 3-manifold $(M,\xi)$, the {\em depth}  of $\mathcal T$, $\dep(\mathcal{T};M)$,  is the minimum number of intersection of an overtwisted disk in $M$ and $\mathcal{T}$. It was  introduced by Baker and Onaran \cite{BO} and measures {\em non-looseness} of $\mathcal{T}$. 
As an application of Theorem~\ref{theorem:non-Nrv-OT} we advance the study of the depth:

\noindent {\bf Corollary~\ref{cor:depth}.} 
{\em
Let $(S,\phi)$ be an open book supporting an overtwisted contact structure, $L$ be a closed braid in $M_{(S,\phi)}$, and $B_{(S,\phi)}$ be the binding. 
\begin{itemize}
\item[(a)] $\dep(B_{(S,\phi)};M_{(S, \phi)})\leq  \min\{N\: | \: \phi \mbox{ is } N\mbox{-twist left-veering}\}$
\item[(b)] $\dep(B_{(S, \phi)};M_{(S, \phi)}\setminus L) \leq  \min\{N\: | \: \phi_L \mbox{ is } (N,0)\mbox{-twist left-veering} \}$
\item[(c)] $\dep(L;M_{(S, \phi)})\leq \min\{k\: | \: \phi_L \mbox{ is } (N,k)\mbox{-twist left-veering for some }N\}$
\item[(d)] $\dep(L\cup B_{(S,\phi)};M_{(S, \phi)})\leq \min\{N+k\: | \: \phi_L \mbox{ is } (N,k)\mbox{-twist left-veering}\}$
\end{itemize}
}

In \cite{IK2, IK-cover} we showed that the inequalities (a),(b) and (d) in Corollary ~\ref{cor:depth} become equalities when the depth is one;
\begin{itemize}
\item $\dep(B_{(S,\phi)};M_{(S,\phi)})=1$ 
if and only if $\phi$ is non-right-veering (i.e. $1$-twist left-veering), 
and 
\item 
$\dep(L\cup B_{(S,\phi)};M_{(S,\phi)})=\dep(B_{(S,\phi)};M_{(S,\phi)} \setminus L) = 1$ if and only if $\phi_L$ is non-quasi-right-veering (i.e., $(1,0)$-twist left-veering). 
\end{itemize}

We extend this result to depth $2$: 

\noindent {\bf Theorem~\ref{prop:depth2}.} 
{\em
Suppose that $\xi_{(S, \phi)}$ is overtwisted. Let $L$ be a closed braid in $M_{(S, \phi)}.$
\begin{itemize}
\item
$\dep( B_{(S,\phi)};M_{(S, \phi)})= 2$ if and only if $\phi$ is right-veering and 2-twist left-veering.
\item
$\dep(B_{(S,\phi)};M_{(S, \phi)} \setminus L)= 2$ if and only if $\phi_L$ is 
quasi-right-veering 
and $(2,0)$-twist left-veering.
\end{itemize}
}

Our boundary based region $R(\Gamma, \phi(\Gamma))$ and Wand's overtwisted region share 
some common properties. See Figure~\ref{comparison-table} for comparison.  
The following theorem shows 
exactly when they coincide.  
\begin{figure}[h]
\begin{tabular}{|l|c|c|} \hline
&{\it overtwisted region} & {\it boundary based region of} \\ 
& & {\it twist-left-veering monodromy} \\ \hline
$2N$-gon & yes & yes \\ \hline
corners alternate between $\partial S$ and $\Int(S)$ & yes & yes \\ \hline
edges alternate $\Gamma$ and $\phi(\Gamma)$ & yes & yes \\ \hline
unique & yes & yes \\ \hline  
embedded {\tiny (Definition~\ref{def-of-R})} & yes & not required \\ \hline
$\Int(\Gamma)\cap\Int(\phi(\Gamma)) \stackrel{?}{=}$ non-base corners of $R$ & yes & allowed to be $\subseteq$ \\ \hline
punctures allowed & no & yes \\ \hline
interpolating arc systems used & no & yes \\ \hline
\end{tabular}
\caption{Comparison of overtwisted region and boundary based region.}
\label{comparison-table}
\end{figure}

\noindent {\bf Theorem~\ref{thm:R=A}.} 
{\em
Let $\Gamma$ be an $N$-arc system with $N\geq 2$ such that the boundary based region $R(\Gamma,\phi(\Gamma))$ exists.  
Then  $R(\Gamma,\phi(\Gamma))$ is an overtwisted region if and only if
$\phi^{tw}(\Gamma) \ll_\ri \Gamma$, $\Int(\phi^{tw}(\Gamma))\cap\Int(\Gamma)=\emptyset$, and 
$R(\Gamma,\phi(\Gamma))$ is embedded. 
{\rm (A parallel statement holds for $N=1$, see Proposition~\ref{prop:R=A}.)}
}

In particular, by Theorem~\ref{theorem:non-Nrv-OT} and Wand's inconsistency criterion we conclude: 

\noindent {\bf Corollary~\ref{cor:stably-LV}.} 
{\em
An open book $(S,\phi)$ supports an overtwisted contact structure if and only if a stabilization of $(S,\phi)$ is $N$-twist left-veering for some $N$.
}

Note that in Corollary~\ref{cor:stably-LV} destabilizations are not required. 
This makes a sharp contrast to the non-right-veering criterion of Honda, Kazez and Mati\'c, which states that an open book $(S, \phi)$ is overtwisted if and only if it is stably equivalent to some non-right-veering open book $(S', \phi')$. 
That is, $(S, \phi)$ and $(S', \phi')$ are related to each other by a sequence of stabilizations and destabilizations (because every open book becomes right-veering after a number of stabilizations \cite[Proposition 6.1]{HKM}).

Theorem~\ref{theorem:non-Nrv-OT} also has an application to the study of contact cyclic branched covers. 
So far, we have been requiring arc systems end at $\partial S$. 
If we allow arc systems end at $\partial S \cup P$ instead, then we can define a similar ordering that we denote by $\ll_\ri^{\partial+P}$.  
It turned out that $\phi^{tw}_L(\Gamma) \ll^{\partial + P}_{\ri} \Gamma$ is a weaker condition than $\phi^{tw}_L(\Gamma) \ll_{\ri} \Gamma$. 
Using this weaker ordering $\ll_\ri^{\partial+P}$ in the place of $\ll_\ri$, we define a notion of \emph{weakly $(N,k)$-twist left-veering}.
We show that weakly $(N,0)$-twist left-veering guarantees virtual looseness of the closed braid $L$ as stated in Theorem~\ref {theorem:virtually-loose}.

\noindent {\bf Theorem~\ref{theorem:virtually-loose}.} 
{\em
Let $L$ be a closed braid with respect to an open book $(S, \phi)$. If $\phi_L$ is weakly $(N, 0)$-twist left-veering; namely, there is an $N$-arc system $\Gamma \in \A$ such that 
\[ \phi^{tw}_L(\Gamma) \ll^{\partial + P}_{\ri} \Gamma \mbox{ in } \AP\]
and the associated boundary based region $R(\Gamma, \phi(\Gamma))$ is non-punctured, then $L$ is virtually loose. 
}

Figure \ref{table:summary2} below would help us to systematically understand four results on looseness and virtually looseness. 

\begin{figure}[h]
\begin{center}
\begin{tabular}{|c|c|} 
\hline
single arc $\gamma$ & arc system $\Gamma$ 
\\ \hline
& 
\\
\cite[Theorem 4.1]{IK-qveer} & Theorem~\ref{theorem:non-Nrv-OT} 
\\
& 
\\
$\phi_L(\gamma) \ll_\ri \gamma$ (non-quasi-rightveering) & $\phi^{tw}_L(\Gamma) \ll_\ri \Gamma$ with non-punctured $R(\Gamma, \phi_L(\Gamma))$ 
\\ 
& 
($(1,0)$-twist left-veering)
\\
&
\\
$\Rightarrow$  $L$ is loose & $\Rightarrow$  $L$ is loose
\\
& 
\\ \hline
&
\\
\cite[Corollary 5.7]{IK-branch} & Theorem~\ref{theorem:virtually-loose}
\\
& 
\\
$\phi_L(\gamma) \prec_\ri \gamma$  (non-rightveering) & $\phi^{tw}_L(\Gamma) \ll_\ri^{\partial + P} \Gamma$ with non-punctured $R(\Gamma, \phi_L(\Gamma))$
\\ 
& (weakly $(N,0)$-twist left-veerting)
\\
&
\\
$\Rightarrow L$ is virtually loose & $\Rightarrow L$ is virtually loose
\\
&
\\
\hline
\end{tabular}
\caption{Four results on looseness and virtual looseness of transverse links} 
\label{table:summary2}
\end{center}
\end{figure}

In Figure \ref{table:summary3} below definitions of various left-veering type notions are summarized. 

\begin{figure}[h]
\begin{center}
\begin{tabular}{|c|c|c|} 
\hline
Notion & Definition & Property \\
\hline
& $\exists$ arc $\gamma$ s.t. & Overtwisted \ ($P=\emptyset$)\\
Non-right-veering & $\phi(\gamma) \prec_{\ri} \gamma$ & 
\\
& ($\phi_L(\gamma) \prec_{\ri} \gamma$) & Virtually loose  \ ($P \neq \emptyset$)\\
\hline
 &  $\exists$ arc $\gamma$ s.t. &\\
 Non-quasi-right-veering & $\phi_L(\gamma) \ll_{\ri} \gamma$ & Loose\\
 & &\\
\hline
 &  $\exists$ $N$-arc system $\Gamma$ s.t. &\\
$N$-twist-left-veering  & $\phi^{tw}(\Gamma) \ll_{\ri} \Gamma$ & Overtwisted\\
 & &\\
\hline
 &  $\exists$ $N$-arc system $\Gamma$ s.t. &\\
$(N,0)$-twist-left-veering  & $\phi_L^{tw}(\Gamma) \ll_{\ri} \Gamma$, & Loose \\
&  $R(\Gamma,\phi_L(\Gamma))$ contains no punctures&\\
\hline
 &  $\exists$ $N$-arc system $\Gamma$ s.t. &\\
Weakly $(N,0)$-twist-left-veering  & $\phi_L^{tw}(\Gamma) \ll_{\ri}^{\partial + P} \Gamma$, & Virtually loose \\
 &$R(\Gamma,\phi_L(\Gamma))$ contains no punctures &\\
\hline
\end{tabular}
\caption{Summary of definitions of various left-veering notions and relation to overtwisted and loose properties.} 
\label{table:summary3}
\end{center}
\end{figure}

The paper is organized as follows. 
In Section~\ref{section:preliminary} we recall basic definitions including the two orderings $\prec_\ri$ and $\ll_\ri$ of arcs in $S$ and non-quasi-right-veering monodromies.  
In Section~\ref{section3} we extend the definition of the two orderings to $N$-arc systems. 
We further define a boundary based region $R(\Gamma, \Gamma')$, left twist $\phi^{tw}(\Gamma)$ of $\Gamma$ and finally introduce the notion of $(N,k)$-twist left-veering.  
In Section~\ref{section4} we prove the main result Theorem~\ref{theorem:non-Nrv-OT}. 
In Section~\ref{section5} we discuss applications of Theorem~\ref{theorem:non-Nrv-OT} to the depths of transverse links. 
In Section~\ref{section6}, relation between $N$-twist left veering and inconsistency is discussed.
In Section~\ref{section7} we introduce $\ll_\ri^{\partial + P}$ and prove Theorem~\ref{theorem:virtually-loose}.

\section*{Acknowledgement}
T.I is partially supported by JSPS KAKENHI Grant Numbers 19K03490, 16H02145. KK was partially supported by Simons Foundation Collaboration Grants for Mathematicians. 

\section{Preliminaries}\label{section:preliminary}  

\subsection{Open books and closed braids}\label{section2.1}

We review the relation between open books and contact 3-manifolds, and the relation between closed braids and transverse links.

Let $S$ be an oriented compact surface with non-empty boundary and $P = \{p_1,\ldots,p_n\}$ be a (possibly empty) set of $n$ distinct interior points of $S$. 
The mapping class group $\MCG(S,P)$ is the group of isotopy classes of orientation preserving diffeomorphisms of the surface $S$ fixing $\partial S$ pointwise and fixing $P$ setwise. In the following, $\MCG(S)=\MCG(S,\emptyset)$.

An \emph{(abstract) open book} is a pair $(S,\phi)$ of a surface $S$ and a diffeomorphism $\phi\in\Diff(S, \partial S)$ (or a mapping class $\phi \in \MCG(S)$). Throughout, for simplicity a diffeomorphism and its mapping class are denoted by the same symbol if no confusions occur. 

For an open book $(S,\phi)$ let $M_{(S,\phi)}$ be a closed 3-manifold defined  by
\begin{equation}\label{def of M_(S phi)}
M_{(S,\phi)}:= S \times[0,1] \slash \sim 
\end{equation}
where $\sim$ is the equivalence relation defined by 
\[ 
(x,1)\sim(\phi(x),0) \ \ \mbox{for all } x \in S, \quad (x,t)\sim (x,s) \ \ \mbox{for all } x \in \partial S \mbox{ and }  t,s \in [0,1].
\] 
We denote the quotient map by $\Pi:S\times[0,1] \rightarrow M_{(S,\phi)}=S \times[0,1] \slash \sim$. 
The \emph{binding} $B$ is defined by 
$$
B=B_{(S, \phi)}:= \Pi(\partial S \times \{t\}) \subset M_{(S,\phi)}
$$
which does not depend on $t\in[0,1]$ since $\phi$ is identity near $\partial S$.  The binding $B$ is a fibered link in $M_{(S,\phi)}$ with natural fibration 
\[
\begin{array}{cccc}
\pi=\pi_{(S,\phi)} : & M_{(S,\phi)} \setminus B_{(S,\phi)} &\longrightarrow & S^{1}=[0,1]\slash 0 \sim 1\\
& \rotatebox{90}{$\in$} & &\rotatebox{90}{$\in$}\\
& \Pi(x, t)   & \longmapsto & t
\end{array}
\]
For each $t \in [0,1]$ the closure of the fiber $S_t:=\overline{\pi^{-1}(t)}$ is called a \emph{page} of the open book. Note that $S_0$ and $S_1$ represent the same page in $M_{(S, \phi)}$. 

For $t \in (0,1)$ let $p_t=\Pi^{-1}|_{S_t}: S_t \to S\times \{t\}=S$ be the canonical diffeomorphism arising from the definition (\ref{def of M_(S phi)}). To define $p_0:S_0 \rightarrow S$, for $x \in S_0$ we choose a sequence of points $\{x_n\in S_{t_n} \subset M_{(S,\phi)}\: | \: 0<t_{n+1}<t_n\}$ that converges to $x$, and define $p_{0}(x)=\lim_{n \to \infty} p_{t_n}(x_n)$ (this is well-defined and independent of a choice of $\{x_n\}$). The diffeomorphism $p_1:S_1 \rightarrow S$ is defined similarly. Thus, $p_0(x)=\phi \circ p_1(x)$ if $x \in S_0=S_1\subset M_{(S, \phi)}$.

A contact structure $\xi=\ker\,\alpha$ on $M_{(S,\phi)}$ is \emph{supported} by the open book $(B_{(S, \phi)}, \pi)$ if $d\alpha>0$ on every page $S_t$ and $\alpha>0$ on $B_{(S, \phi)}$. 
Up to isotopy there exists a unique contact structure $\xi_{(S,\phi)}$ on $M_{(S,\phi)}$ supported by $(B_{(S, \phi)}, \pi)$ \cite{Giroux}. 
If $(M, \xi)$ and $(M_{(S,\phi)},\xi_{(S,\phi)})$ are contactomorphic then we say that $(M, \xi)$ is supported by the open book $(S,\phi)$ or, $(S,\phi)$ is an open book of $(M, \xi)$.

A link $L$ in $M_{(S,\phi)}$ is a {\em closed braid} with respect to $(S,\phi)$ if $L \subset M_{(S,\phi)} \setminus B_{(S, \phi)}$ and $L$ is positively transverse to every page. 
For a fixed $(S, \phi)$ two closed braids $L_0$ and $L_1$ are called \emph{braid isotopic} if they are isotopic through a continuous family $\{L_t\}_{0\leq t \leq 1}$ of closed braids with respect to $(S, \phi)$.

Suppose that $(M, \xi)$ is supported by $(S, \phi)$. 
We say that a transverse link $\mathcal T$ in $(M, \xi)$ is {\em represented by} a closed braid $L$ with respect to $(S, \phi)$ if a contactomorphism $(M_{(S,\phi)},\xi_{(S,\phi)}) \to (M, \xi)$ takes $L$ to $\mathcal T$ (up to transverse isotopy). 
In fact, due to Bennequin \cite{Ben} and Pavelescu \cite{pav1, pav}, given a contact manifold $(M, \xi)$ and its open book $(S, \phi)$, every transverse link $\mathcal T$ in $(M, \xi)$ is represented by some closed braid $L$ with respect to $(S, \phi)$ and such $L$ is unique up to braid isotopy, positive braid stabilization and positive braid destabilization.

Let $L \subset M_{(S,\phi)}$ be a closed braid with respect to $(S,\phi)$. Take a collar neighborhood $\nu(\partial S)$ so that $\phi|_{\nu(\partial S)}=id_{\nu(\partial S)}$, and move $L$ by braid isotopy so that $P:= p_0(L\cap S_0) \subset \nu(\partial S)$. Then $\phi$ is regarded as a diffeomorphism $(S,P)\rightarrow (S,P)$ hence it gives an element $j(\phi) \in \MCG(S,P)$.
By cutting $M_{(S,\phi)}$ along the page $S_0$ we get a cylinder $S\times[0,1]$ and $L$ gives rise to an element $\beta_L$ of the surface braid group $B_n(S)$. We define the the \emph{distinguished monodromy} of $L$ by 
\[ \phi_L=j(\phi)\ i(\beta_L)  \]
where $i$ is the push map in the generalized Birman exact sequence \cite[Theorem 9.1]{FM}
\begin{equation*}
\label{eqn:Birman}
1 \rightarrow B_n(S)  \stackrel{i}{\rightarrow} \MCG(S,P)  \stackrel{f}{\rightarrow} \MCG(S) \rightarrow 1.
\end{equation*}

The distinguished monodromy $\phi_L$ is well-defined up to point-changing isomorphism \cite{IK-qveer}: 
If two closed braids $L$ and $L'$ with respect to $(S, \phi)$ are braid isotopic, then there is a point-changing isomorphism $\Theta: \MCG(S,P)\rightarrow \MCG(S,P')$ where $P=p_0(S_0\cap L)$, $P'=p_0(S_0\cap L')$ such that $\Theta(\phi_L)=\phi_{L'}$. 
Here the point-changing isomorphism $\Theta$ is an isomorphism defined by $\Theta([\psi])= [\theta^{-1} \circ \psi \circ \theta] $ for some orientation-preserving diffeomorphism $\theta: (S, P')\to (S, P)$ such that $\theta|_{\partial S}=id_{\partial S}$ and $\theta$ is isotopic to $id_S$ if we forget the sets of marked points $P$ and $P'$. When $P=P'$, this simply means that $\phi_L$ and $\phi_{L'}$ are conjugate in $\MCG(S,P)$.

\subsection{Strong right-veering ordering and quasi-right-veering}

We review the right-veering orderings $\prec_\ri$ and $\ll_\ri$ of arcs and the definition of quasi-right-veering. 

Take a base point $v \in \partial S$.
Let $\mathcal{A}_{v}(S,P)$ be the set of isotopy classes of oriented properly embedded arcs $\gamma:[0,1] \rightarrow S \setminus P$ satisfying $\gamma(0) = v$ and $\gamma(1) \in \partial S \setminus \{v\}$. Here we allow arcs to be boundary-parallel and by an isotopy we mean isotopy fixing $\partial S$. Let $\underline{\gamma}$ denote the arc $\gamma$ with the reversed orientation so that $\underline{\gamma}(t)=\gamma(1-t)$.

We call an element of $\mathcal{A}_{v}(S,P)$ an \emph{arc} based on $b$.
Abusing the notation, by an arc $\gamma \in \mathcal{A}_{v}(S,P)$ we will mean three different objects.
\begin{itemize}
\item a map $\gamma: [0,1] \rightarrow S \setminus P$,
\item The image of the map (viewed as a submanifold) $\gamma([0,1]) \subset S \setminus P$.
\item The isotopy class of the submanifold $\gamma$. 
\end{itemize}

We say that arcs $\alpha$ and $\beta$ intersect \emph{efficiently} if they realize the geometric intersection number. 
We will always assume that all arcs intersect pairwise efficiently. In particular, when $\alpha$ and $\beta$ have the same isotopy class then $\Int(\alpha)\cap \Int(\beta) = \emptyset$.

The following orderings are important in 3-dimensional contact topology.

\begin{definition}[Right-veering ordering $\prec_{\sf right}$]
\label{def:<single-arc}
For arcs $\alpha,\beta \in \mathcal{A}_{v}(S,P)$, we denote $\alpha \prec_{\sf right} \beta$ if $\alpha \neq \beta$ and the arc $\beta$ lies on the right side of $\alpha$ in a small neighborhood of the base point $v$. 
\end{definition}

\begin{definition}[Strong right-veering ordering $\ll_\ri$ \cite{IK-qveer}] 
\label{def:<<single-arc}
For arcs $\alpha,\alpha' \in \mathcal{A}_{v}(S,P)$, we denote $\alpha\disj\alpha'$ if $\alpha\prec_\ri\alpha'$ and $\alpha\cap\alpha'=\{v\}$. 
We denote
$\alpha \ll_{\sf right} \alpha'$ if $\alpha \neq \alpha'$ and there exists a sequence of arcs $\alpha_0,\ldots,\alpha_k \in \mathcal{A}_{v}(S,P)$ such that 
\begin{equation}
\label{eqn:sequence-<<singlearc}
\alpha = \alpha_0 \disj \alpha_1 \disj  \cdots \disj \alpha_k = \alpha'.
\end{equation}
\end{definition}

The relation $\prec_\ri$ is a strict total order and $\ll_\ri$ is a strict partial order (strict in the sense that it is irreflexive; $a \not\prec_{\ri} a$ for all $a$). The relation $\disj$ is not a strict partial order since it is not transitive; $a \disj b$ and $b \disj c \not\Rightarrow a \disj c$.

\begin{remark}
The definitions of $\ll_{\ri}$ in Definition \ref{def:<<single-arc} and \cite{IK-qveer} are slightly different. 
In this paper for the sequence of arcs (\ref{eqn:sequence-<<singlearc}) we require $\alpha_i \cap \alpha_{i+1}=\{v\}$ (so $\alpha_{i}$ and $\alpha_{i+1}$ only shares the common base point $v$), whereas in \cite{IK-qveer} we require a weaker condition $\Int(\alpha_i) \cap \Int(\alpha_{i+1})=\emptyset$ (so we allow $\alpha_{i}$ and $\alpha_{i+1}$ have the same start point and end point).

However, as for the definition of quasi-right-veering, Definition \ref{def:not-quasi-right-veering} below and the one given in \cite{IK-qveer} are equivalent. This is because when $\alpha$ and $\beta$ share the same endpoint, we can perturb the endpoints of arcs to make them distinct as long as $\alpha \cup \beta$ does not cut a punctured disk (bigon).
\end{remark}

When $P=\emptyset$ the orderings $\prec_{\sf right}$ and $\ll_{\sf right}$ are the same, which can be proved using a work of Honda, Kazez and Mati\'c \cite[Lemma 4.1]{HKM}. However, when $P\neq \emptyset$, $\prec_{\sf right}$ and $\ll_{\sf right}$ are different as studied in \cite{IK-qveer}.

The mapping class group $\MCG(S,P)$ acts on $\mathcal{A}_{v}(S,P)$ and the action preserves both $\prec_{\sf right}$ and $\ll_{\sf right}$. Using the total ordering $\prec_{\sf right}$ the right-veering property is defined as follows.

\begin{definition}[Right-veering \cite{HKM}]
\label{def:right-veering}
An element $\phi\in \MCG(S,P)$ is \emph{right-veering} if for any $v \in \partial S$ and any arc $\alpha \in \mathcal{A}_v(S,P)$, $\alpha \preceq_{\sf right} \phi(\alpha)$; that is, $\alpha \prec_{\sf right} \phi(\alpha)$ or $\alpha = \phi(\alpha)$.
\end{definition}

For our purpose, however, it is much more convenient to accept \emph{non-right-veering} as the basic concept. We view right-veering as {\em not non-right-veering}. Definition~\ref{def:right-veering} is rephrased as follows: 

\begin{definition}[Non-right-veering]
\label{def:not-right-veering}
An element $\phi\in \MCG(S,P)$ is \emph{non-right-veering} if there exists $v \in \partial S$ and an arc $\alpha \in \mathcal{A}_v(S,P)$ such that $\phi(\alpha) \prec_{\sf right}\alpha$. Otherwise, we say that $\phi$ is \emph{right-veering}.
\end{definition}

Using the strong right-veering ordering $\ll_{\sf right}$ in the place of   $\prec_{\sf right}$, we define quasi-right-veering.

\begin{definition}[Non-quasi-right-veering \cite{IK-qveer}]
\label{def:not-quasi-right-veering}
An element $\phi \in \MCG(S,P)$ is \emph{non-quasi-right-veering} if there exists $v \in \partial S$ and an arc $\alpha \in \mathcal{A}_v(S,P)$ such that $\phi(\alpha) \ll_{\sf right}\alpha$. 
Otherwise, we say that $\phi$ is \emph{quasi-right-veering}.
\end{definition}

Since $\ll_{\sf right}$ and $\prec_{\sf right}$ are the same when $P=\emptyset$, $\phi \in \MCG(S)$ is right-veering if and only if it is quasi-right-veering. On the other hand, when $P\neq \emptyset$, $\phi$ is quasi-right-veering if $\phi$ is right-veering \cite[Proposition 3.14]{IK-qveer}, but the converse is not true in general.

The following is a list of properties of non-quasi-rightveering.

\begin{theorem}\label{thm:previous}
Let $(M, \xi)$ be a closed contact 3-manifold and $\mathcal T$ be a transverse link in $(M, \xi)$. 
\begin{enumerate}
\item 

$(M, \xi)$ is overtwisted if and only if there exists a non-right-veering open book $(S, \phi)$ that supports $(M, \xi)$.  
\cite{HKM} 

\item 

$\mathcal T$ is loose if and only if there exist an open book $(S, \phi)$ supporting $(M, \xi)$ and a closed braid representative $L$ of $\mathcal T$ with respect to $(S, \phi)$  such that the distinguished monodromy $\phi_L$ is non-quasi-right-veering. \cite{IK-qveer} 

\item 

If $\mathcal T$ is represented by a closed braid $L$ with respect to an open book $(S, \phi)$ whose distinguished monodromy $\phi_L$ is non-right-veering then $\mathcal T$ is virtually loose; that is, there is a finite covering of $M \setminus \mathcal T$ on which the lifted contact structure is overtwisted.  \cite{IK-branch} 
\end{enumerate}
\end{theorem}

\section{Right-veering ordering on arc systems and $N$-twist left-veering}
\label{section3}

\subsection{Right-veering ordering on arc systems}

In this section we extend the strong right-veering ordering $\ll_\ri$ to  arc systems.

\begin{definition}[$N$-arc system]
\label{defn:Narcsystem}
Let $\mathcal{B}=\{v_1,\ldots,v_N\} \subset \partial S$ be distinct $N$ boundary points. Let 
$$ 
\mathcal{A}_{\mathcal{B}}(S,P) = 
\left\{\Gamma=(\gamma_1,\ldots,\gamma_N) \in \mathcal{A}_{v_1}(S, P)\times \cdots \times \mathcal{A}_{v_N}(S,P)\: | \:
 \gamma_i \cap \gamma_j = \emptyset \ \ (i \neq j) 
\right\}. 
$$ 
We call an element of $\mathcal{A}_{\mathcal{B}}(S, P)$ an \emph{$N$-arc system} of $(S, P)$ based on $\mathcal{B}$. 
\end{definition}

We may abuse the symbol $\Gamma$ for different objects such as:
\begin{itemize}
\item A collection of maps $\{\gamma_i:[0,1]\to S \ | \ i=1,\dots N \}$. 
\item The submanifold $\gamma_1 \cup \gamma_2 \cup \cdots \cup \gamma_N \subset S$
\item The isotopy class (rel. $\partial S$) of the submanifold $\gamma_1 \cup \gamma_2 \cup \cdots \cup \gamma_N \subset S$.
\end{itemize}
We denote $\Gamma(1):=\{\gamma_1(1), \gamma_2(1), \ldots, \gamma_N(1)\}=\partial \Gamma \setminus \mathcal B.$

We naturally generalize $\prec_{\sf right}$ and $\ll_{\sf right}$ of Definitions~\ref{def:<single-arc} and \ref{def:<<single-arc} as follows:

\begin{definition}[Right-veering ordering $\prec_{\sf right}$]
For $N$-arc systems $\Gamma=(\gamma_1,\ldots,\gamma_N)$ and $\Gamma'=(\gamma'_1,\ldots,\gamma'_N) \in  \A$, we denote $\Gamma \prec_{\sf right} \Gamma'$ if and only if $\Gamma \neq \Gamma'$ and $\gamma_i \preceq_{\sf right} \gamma'_i$ for all $i=1,\ldots,N$. 
\end{definition}

\begin{definition}[Strong right-veering ordering $\ll_\ri$]\label{def:SO} 
Let $\Gamma$ and $\Gamma' \in  \A$.
\begin{itemize}
\item
We denote $\Gamma \disj \Gamma'$ if $\Gamma \prec_\ri \Gamma'$ and $\Gamma\cap\Gamma'=\B.$ 
\item
We denote $\Gamma \ll_\ri \Gamma'$ if $\Gamma \neq \Gamma'$ and there exists a finite sequence of $N$-arc systems $\Gamma_0,\dots,\Gamma_k \in \A$ such that 
$\Gamma = \Gamma_0 \disj \Gamma_1 \disj  \cdots \disj \Gamma_k = \Gamma'. 
$
\end{itemize}
\end{definition}

While $\ll_\ri$ is an strict partial ordering, $\disj$ is not.

\begin{example}
See Figure \ref{fig:arc-order}. 
Sketch (i) depicts single arcs $\alpha$ and $\beta$ satisfying $\alpha\ll_{\sf right} \beta.$
Sketch (ii) shows $(\alpha_1,\alpha_2) \disj(\gamma_1,\gamma_2)\disj (\beta_1,\beta_2)$; thus, $(\alpha_1,\alpha_2) \ll_{\sf right} (\beta_1,\beta_2)$.
\begin{figure}[htbp]
\includegraphics*[width=100mm,bb=152 583 447 716]{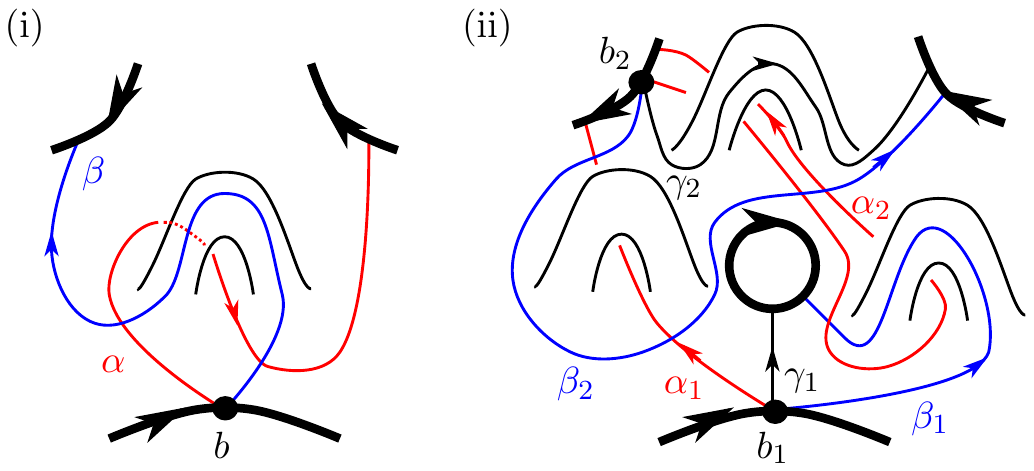}
\caption{(i) $N=1$ and $\alpha \ll_\ri \beta$. 
(ii) $N=2$ and $(\alpha_1,\alpha_2) \ll_{\sf right} (\beta_1,\beta_2)$.} 
\label{fig:arc-order}
\end{figure} 
\end{example}

If $N>1$ both $\prec_{\sf right}$ and $\ll_{\sf right}$ are strict partial orderings but they are distinct since $\Gamma \prec_{\sf right} \Gamma'$ does not imply $\Gamma \ll_{\sf right} \Gamma'$ even if $P=\emptyset$. 
Indeed, $\gamma_i \ll_{\sf right} \gamma'_{i}$ for all $i=1,\ldots,N$ (which is stronger than $\Gamma\prec_\ri \Gamma'$ when $P \neq \emptyset$) does not imply $\Gamma \ll_{\sf right} \Gamma'$ as the next example shows.

\begin{example}
\label{example:prec-does-not-imply-ll}
For $\Gamma=(\gamma_1,\gamma_2),\Gamma'=(\gamma'_1,\gamma'_2) \in \mathcal{A}_{\{b_1,b_2\}}(S)$ in Figure \ref{fig:arc-order-example}, we have $\gamma_1 \ll_{\sf right}\gamma'_1$ and $\gamma_2 \ll_{\sf right}\gamma'_2$. 
That is, $\Gamma\prec_\ri\Gamma'$. 
However, $\Gamma \not \ll_{\sf right} \Gamma'$ because due to the shaded region enclosed by $\Gamma \cup \Gamma'$, any arc system $\Gamma''=(\gamma''_1,\gamma''_2)\in\mathcal{A}_{\{b_1,b_2\}}(S)$ with $\gamma_i \prec_{\sf right} \gamma''_i \prec_{\sf right} \gamma'_i$ $(i=1,2)$, $\gamma''_1$ must intersect $\gamma_2\cup\gamma'_2$ and $\gamma''_2$ must intersect $\gamma_1\cup \gamma'_1$.
\end{example}
\begin{figure}[htbp]
\includegraphics*[width=30mm,bb=254 560 355 722]{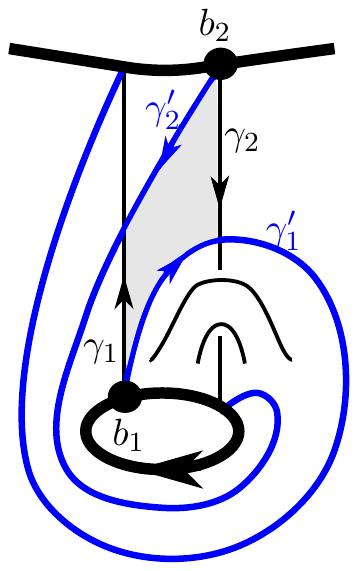}
\caption{
$\Gamma\prec_\ri \Gamma'$ (moreover, $\gamma_i \ll_{\ri} \gamma'_i$ for every $i$) but $\Gamma\not\ll_\ri\Gamma'$.
The orderings $\prec_{\sf right}$ and $\ll_{\sf right}$ are different when $N>1$.} 
\label{fig:arc-order-example}
\end{figure}

\subsection{$N$-twist left-veering}
In this section we define $N$-twist left-veering as a generalization of non-right-veering.

\begin{definition}[Boundary based region]
\label{def-of-R} 
Suppose that $N$-arc systems $\Gamma,\Gamma' \in \A$ efficiently intersect. 
The \emph{boundary based region} $R=R(\Gamma,\Gamma')$ for $\Gamma$ and $\Gamma'$ is an embedded $2N$-gon in $S$ with the boundary $\partial R \subset \Gamma \cup \Gamma'$ (see Figure \ref{fig:boundary-obstruction}, \ref{fig:zu-boundary-twist}) such that: 
\begin{itemize}
\item[(i)] The interior of $R$ is disjoint from $\Gamma'$
\item[(ii)] The orientation of $\partial R$ agrees with that of $\Gamma'$ and disagrees with that of $\Gamma$.
\item[(iii)] 
The corners of $R$ are read $v_1, q_1, v_2, q_2,\ldots,v_N, q_N$ with respect to the orientation of $\partial R$ where $\{v_1,\ldots, v_N\} =\mathcal B \subset \partial S$ and $\{q_1,\ldots,q_N \} \subset (\Gamma\cap \Gamma') \setminus \mathcal{B}.$
\end{itemize}
\end{definition}

Here are some terminologies we use: 

\begin{itemize}
\item
Renumbering the arcs of $\Gamma=(\gamma_1,\ldots,\gamma_N)$ and $\Gamma'=(\gamma'_1,\ldots,\gamma'_N)$ we may assume that $v_i=\gamma_i(0)$ and $q_i \in \gamma'_i \cap \gamma_{i+1}$ for $i=1,\ldots,N$ (mod $N)$. 
Corners $v_1,\ldots,v_N$ are called {\em base corners} and depicted by black dots $\bullet_{v_i}$ in figures, and  
$q_1,\ldots,q_N$ are called {\em non-base corners} and depicted by hollow circles $\circ_{q_i}$. 

\item
When $\Gamma$ and $\Gamma'$ do not form any boundary based region we denote $R(\Gamma,\Gamma')=\emptyset$.

\item If $\Int(R(\Gamma,\Gamma'))\cap \Gamma =\emptyset$ then we say that $R(\Gamma,\Gamma')$ is {\em embedded}.

\item 
If $R(\Gamma, \Gamma') \neq \emptyset$ contains $k$ puncture points then it is said to have type $(N, k)$.
\item 
When $N=1$, we say that $R(\Gamma, \Gamma')=R(\gamma, \gamma')$ has type $(1,0)$ if and only if $\gamma' \prec_\ri \gamma$.
\end{itemize}

Here are some remarks we want to highlight: 
\begin{itemize}
\item 
All the non-base corners $q_i$ are negative intersections of $\Gamma$ and $\Gamma'$.

\item
Given $\Gamma,\Gamma'$, the boundary based region is unique if it exists. 

\item
Although $\Gamma'$ and $\Int(R)$ are disjoint, $\Gamma$ may intersect $\Int(R)$.

\item 
If $R(\Gamma, \Gamma')$ has type $(1,0)$ then $R(\Gamma, \Gamma')=\emptyset$. 
 
\item 
For $N=1$, if $\Gamma=(\gamma)$ and $\Gamma'=(\gamma')$ form a boundary based region $R(\Gamma, \Gamma')$ then it is a punctured bigon as depicted in Figure \ref{fig:N1case} since $\Gamma$ and $\Gamma'$ intersect efficiently. 
In \cite{IK-qveer}, it is called a \emph{boundary right $P$-bigon} from $\gamma$ to $\gamma'$.
\end{itemize}

With the above remark, boundary based regions can be viewed as generalization of boundary right $P$-bigons.
However,  since a boundary right $P$-bigon may be immersed, it is not always a boundary based region. 
In \cite[Proposition 3.5]{IK-qveer}, it is shown that a boundary right P-bigon serves as an obstruction for $\gamma \ll_\ri \gamma'$. 
Similarly, a boundary based region serves as an obstruction for $\Gamma \ll_{\ri} \Gamma'$:

\begin{figure}[htbp]
\includegraphics*[width=25mm,bb=267 618 343 717]{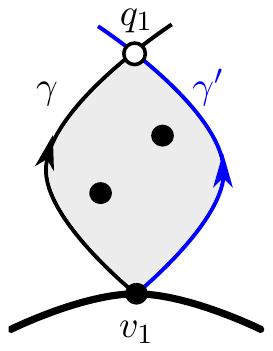}
\caption{A boundary based region where $N=1$. (In \cite{IK-qveer} it is called a boundary right $P$-bigon.)}
\label{fig:N1case}
\end{figure} 

\begin{proposition}
\label{prop:Robstruction}
If $R(\Gamma,\Gamma') \neq \emptyset$ then $\Gamma \prec_\ri \Gamma'$ but $\Gamma \not\ll_{\sf right} \Gamma'$. 
\end{proposition}

The proof is related to Example \ref{example:prec-does-not-imply-ll}.

\begin{proof}
Assume to the contrary that $R(\Gamma,\Gamma') \neq \emptyset$ and $\Gamma \ll_{\sf right} \Gamma'$. 
Then there exists an interpolating sequence: 
$$\Gamma = \Gamma_0 \disj \Gamma_1 \disj \cdots \disj \Gamma_k = \Gamma'$$
for some $k\geq 2$. 
The $N$-arc system $\Gamma_{k-1}$ cuts $R(\Gamma,\Gamma')$ into smaller pieces. There exist sub $M$-arc systems $\Gamma_{k-1}^{sb} \subset \Gamma_{k-1}$ and $\Gamma'^{sb} \subset \Gamma'$ with $2 \leq M \leq N$ we have $ \emptyset \neq R(\Gamma_{k-1}^{sb}, \Gamma'^{sb}) \subset R(\Gamma, \Gamma'),$  
see Figure~\ref{fig:boundary-obstruction}. 
This contradicts the assumption that $\Gamma_{k-1}\disj \Gamma'$.
\end{proof}

\begin{figure}[htbp]
\includegraphics*[width=120mm,bb=129 592 495 721]{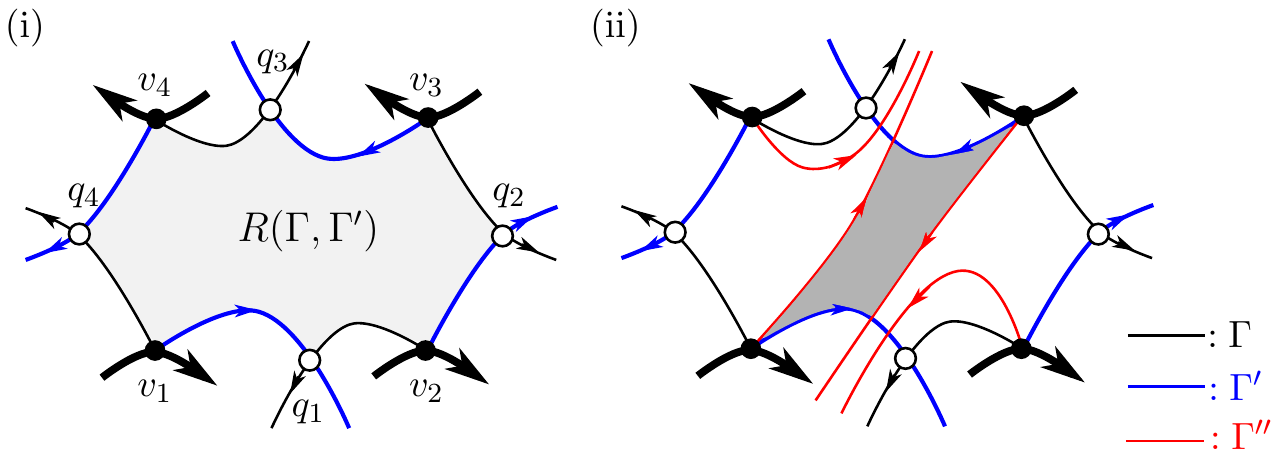}
\caption{(i): Boundary based region $R(\Gamma,\Gamma')$. (ii) When $R(\Gamma,\Gamma')\neq \emptyset$, $\Gamma \prec_{\ri} \Gamma'$ but $\Gamma \not\ll_{\ri} \Gamma'$ since for $\Gamma \prec^{\sf disj} \Gamma'' \prec_{\ri} \Gamma'$, sub arc systems of $\Gamma''$ and $\Gamma'$ form another boundary based region. }
\label{fig:boundary-obstruction}
\end{figure}

In the following, we mainly study boundary based regions $R(\Gamma,\phi(\Gamma))$ where $\Gamma'=\phi(\Gamma)$ for $\phi \in \MCG(S,P)$.

\begin{definition}[Left twist]\label{def:left-twist}
Let $\mathcal{B}=(v_1,\ldots,v_N)$.
For $\Gamma \in \mathcal{A}_{\mathcal{B}}(S, P)$ and $\phi \in \MCG(S, P)$ we define $\phi^{tw}(\Gamma)$ the {\em left-twist} of $\Gamma$ for $\phi$ as follows: 
\begin{itemize}
\item
When $R(\Gamma,\phi(\Gamma)) =\emptyset$ we define $\phi^{tw}(\Gamma)=\phi(\Gamma)$.
\item
When $R(\Gamma,\phi(\Gamma)) \neq \emptyset$, we take $t_i,s_i \in [0,1]$ so that $q_i=\phi(\gamma_i(t_i))=\gamma_{i+1}(s_i)$ and define 
\[ \phi^{tw}(\Gamma)= (\phi^{tw}(\gamma_1), \ldots, \phi^{tw}(\gamma_N))
\]
where 
\[ 
\phi^{tw}(\gamma_i)=\gamma_{i}|_{[0,s_{i-1}]}\ast (\phi\circ \gamma_{i-1})|_{[t_{i-1},1]} \]
start from $v_i$ and go along $\gamma_i$ until reaching $q_{i-1}$ then turn left and switch to $\phi(\gamma_{i-1})$ to the end. Here $\ast$ represents the concatenation of paths.
See Figure \ref{fig:zu-boundary-twist}. 
(When $N=1$ we have $\gamma_i=\gamma_{i+1}$.) 
\end{itemize}
\end{definition}

\begin{figure}[htbp]
\includegraphics*[width=90mm,bb=172 550 463 716]{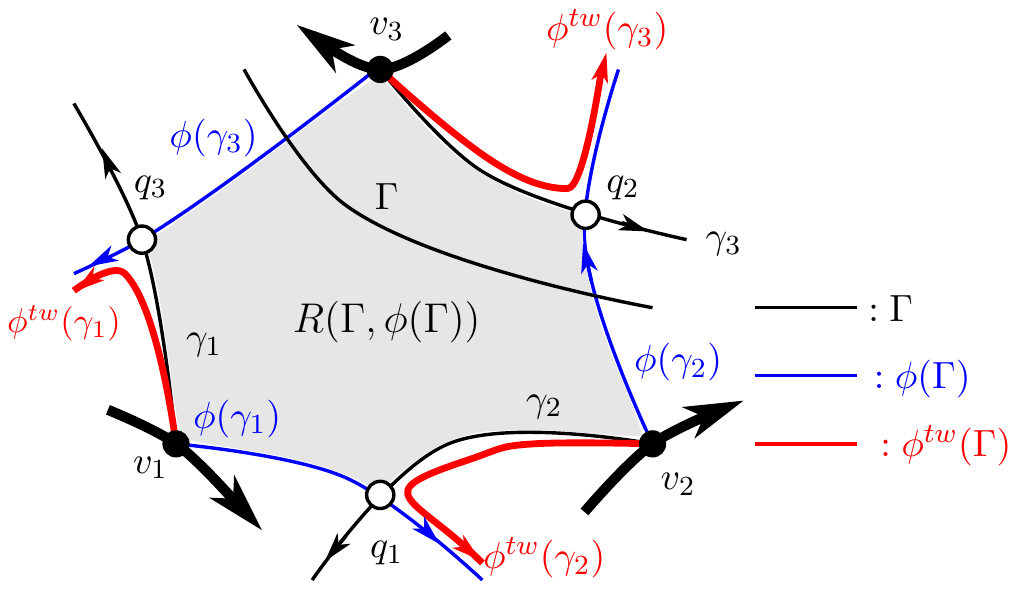}
\caption{The boundary based region $R(\Gamma,\phi(\Gamma))$ and the left twist $\phi^{tw}(\Gamma)$: from each base point $v_i$ walk along $\gamma_i$ until hitting $q_{i-1}$ then turn left and walk along $\phi(\gamma_{i-1})$.} 
\label{fig:zu-boundary-twist}
\end{figure}

Using the left-twist $\phi^{tw}(\Gamma)$ we define $(N,k)$-twist left-veering.

\begin{definition}($(N,k)$-twist left-veering)
An element $\phi \in MCG(S,P)$ is \emph{$(N,k)$-twist left-veering} if 
there exist a set of $N$ base points $\mathcal{B}$ and an $N$-arc system $\Gamma \in \A$ such that
\begin{itemize}
\item $R(\Gamma, \phi(\Gamma))$ has type $(N, k)$, and 
 
\item $\phi^{tw}(\Gamma) \ll_{\sf right} \Gamma.$
\end{itemize}

We say that $\phi \in MCG(S,P)$ is \emph{$N$-twist left-veering} if $\phi$ is $(N,k)$-twist left-veering for some $k$.
\end{definition}

\begin{remark}
\label{remark:N=1}
When $N=1$, we have the following: 
\begin{itemize}
\item 
$\phi \in \MCG(S)$ is non-right-veering if and only if 
$\phi$ is $1$-twist left-veering.
\item
$\phi \in \MCG(S, P)$ is  non-quasi-right-veering if and only if $\phi$ is $(1,0)$-twist left-veering.
\end{itemize}
\end{remark}

\section{Overtwisted disks and twist left-veering monodromies} 
\label{section4}

\begin{theorem}
\label{theorem:non-Nrv-OT}
Let $L$ be a closed braid with respect to $(S, \phi)$ and $\phi_L \in \MCG(S,P)$ be its distinguished monodromy. 
\begin{enumerate}
\item If $\phi \in \MCG(S)$ is $N$-twist left-veering then $(M_{(S,\phi)},\xi_{(S,\phi)})$ is overtwisted, and there is an overtwisted disk $D$ that intersects the binding $B$ at $N$ points.
\item 
If $\phi_L \in \MCG(S, P)$ is $(N,k)$-twist left-veering then there is an overtwisted disk $D$ that intersects the binding $B$ at $N$ points and intersects the closed braid $L$ at $k$ points.
In particular, if $\phi_L$ is $(N,0)$-twist left-veering then $L$ is loose.
\end{enumerate}
\end{theorem}

For the proof, we use {\em open book foliations} introduced by the authors in \cite{IK1-1} and we will assume the readers
are familiar with the definition and basic properties of open book foliations.
See the research monograph \cite{LM} by LaFountain and Menasco for a gentle introduction to the techniques of open book foliations that is central to the new work in this paper. 
Open book foliations had their origins in the work of Birman and Menasco in a series of papers about {\it braid foliations}. See Birman and Finkelstein's article \cite{BF} for a useful guide to the work of Birman and Menasco on braid foliations, and \cite{bm2} for their key paper that is relevant for us. It is the first place where braid foliations were used to solve a major then-open problem in contact topology.

In the proof, we will construct a transverse overtwisted disk $D$ which has $N$ negative elliptic points that are exactly the negative intersection points of $B \cap D$. 
For simplicity, both b-arc $\gamma \subset S_t$ of the open book foliation of $D$ and its image $p_t(\gamma) \subset S$ under  the canonical diffeomorphism $p_t:S_t\to S$ (defined in Section~\ref{section2.1}) are denoted by the same letter $\gamma$. This convention also applies in the proof of Theorem~\ref{prop:depth2}.
The idea of the construction of $D$ can be found in the proof of Theorem 2.4 in \cite{IK1-2}.  
As shown in \cite[Theorem 3]{IK-cover},  see also \cite[Theorem 5.2]{IK-qveer}, shrinking such a transverse overtwisted disk $D$ gives a usual overtwisted disk $D'$ with the geometric intersection $|D'\cap B| = N$ as desired.

\begin{proof}
We prove (2) since (1) is a special case of (2) where $P=\emptyset$ (i.e., the closed braid $L$ is empty).

Assume that $\phi_L \in \MCG(S, P)$ is $(N,k)$-twist left-veering. 

If $(N, k)=(1,0)$ then Remark \ref{remark:N=1} states that $\phi_L$ is non-quasi-right-veering. By \cite[Theorem 4.4]{IK-qveer} $L$ is loose. 
Moreover in the proof, an overtwisted disk that intersects the binding at one point is constructed.

If $N=1$ and $k>0$ then there is an arc $\gamma$ such that the bigon $R(\gamma, \phi_L(\gamma))$ contains $k$ punctures and $\phi_L^{tw}(\gamma) \ll_{\ri} \gamma$.
Suppose that $\phi_L^{tw}(\gamma) = \gamma_0 \disj \gamma_1 \disj \cdots \disj \gamma_k=\gamma$.
Using the interpolating sequence, by exactly the same way as in the proof of \cite[Theorem 4.4]{IK-qveer}, we construct a transverse overtwisted disk until we come to the moment of identifying the b-arcs $\phi_L^{tw}(\gamma)$ in the page $S_0$ and $\gamma$ in the page $S_1$ with the diffeomorphism $\phi_L$. 
See the gray cut disk in Figure~\ref{fig:cut-disk} (1). 
\begin{figure}[htbp]
\includegraphics*[width=110mm,bb=118 400 494 716]{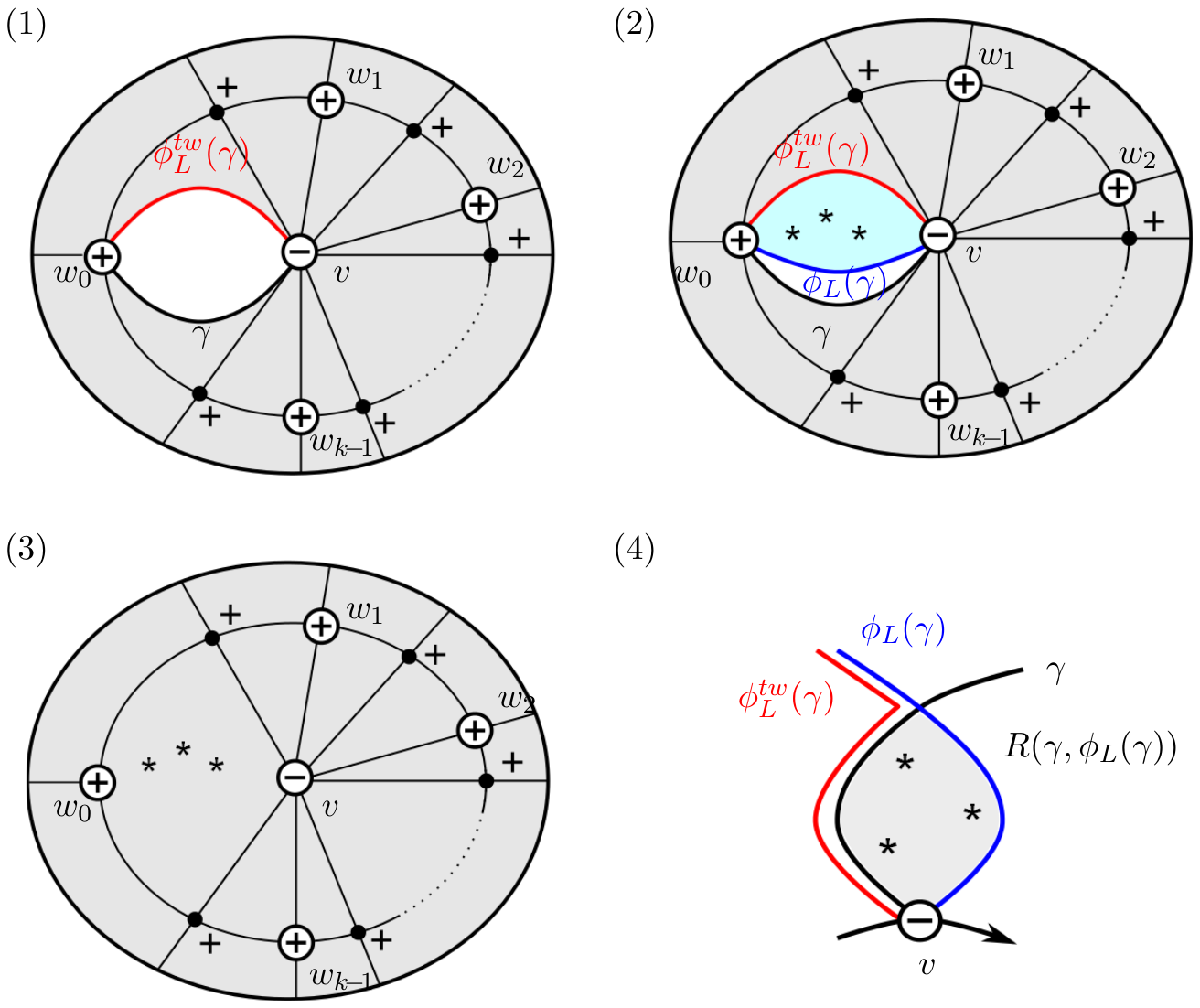}
\caption{Construction of a transverse overtwisted disk for $N=1$ and $k>0$.} 
\label{fig:cut-disk}
\end{figure} 

Clearly $\phi_L^{tw}(\gamma)\neq \phi_L(\gamma)$ in $\A$, which means we cannot identify the b-arcs $\phi_L^{tw}(\gamma)$ and $\gamma$ via $\phi_L$. 
To get around this issue, we 
\begin{itemize}
\item
place the b-arc $\phi_L^{tw}(\gamma)$ in the page $S_{\varepsilon}$ for some small $\varepsilon>0$,  
\item 
place the b-arc $\phi_L(\gamma)$ in the page $S_0$, and 
\item 
between the pages $S_0$ and $S_\varepsilon$ insert a family of b-arcs 
coming from an isotopy from $\phi_L(\gamma)$ to $\phi_L^{tw}(\gamma)$ that sweeps the $k$-punctured bigon $R(\gamma, \phi_L(\gamma))$, see Figure~\ref{fig:cut-disk} (4).
\end{itemize}
The open book foliation of the cut disk is changed to Figure~\ref{fig:cut-disk} (2) where the blue added bigon corresponds to the inserted family of b-arcs. 
Since the isotopy is done by passing the $k$ punctures of $R(\gamma, \phi_L(\gamma))$, the blue bigon is also punctured $k$ times, where the braid $L$ intersects the disk. 
Now we can successfully glue the b-arcs $\phi_L(\gamma)$ in $S_0$ and $\gamma$ in $S_1$ and obtain a transverse overtwisted disk with $k$ punctures, see Figure~\ref{fig:cut-disk} (3).

Next assume that $N\geq 2$. 
Since $\phi_L^{tw}(\Gamma) \ll_{\sf right} \Gamma$ we have an interpolating sequence of $N$-arc systems for some $n$: 
\[ 
\phi_L^{tw}(\Gamma)=\Gamma_0 \disj \Gamma_1 \disj \cdots \disj \Gamma_n=\Gamma.
\]
Since $\phi_L^{tw}(\Gamma)(1)=\Gamma(1)$ (as an unordered set of points) we note that $n\geq 2$. 
Let $\Gamma_j=(\gamma^j_1,\ldots,\gamma^j_N)$.
For $i=1,\ldots,N$ and $j=0,\ldots,n$, let $w^{j}_i:=\gamma_i^j(1)$ be the endpoint of the arc $\gamma^{j}_i$ different from the base point $v_i:=\gamma_i^j(0)$. When $\gamma^j_i=\gamma^{j+1}_{i}$ we place $w_i^j$ and $w_i^{j+1}$ next to each other so that $\gamma^j_i \prec_\ri \gamma^{j+1}_{i}$.
Thus, points $w_i^j$ are pairwise distinct except for $w^{n}_{i-1}=w_{i}^{0}$. See Figure~\ref{figA}.
\begin{figure}[htbp]
\includegraphics*[width=85mm,bb=166 593 447 727]{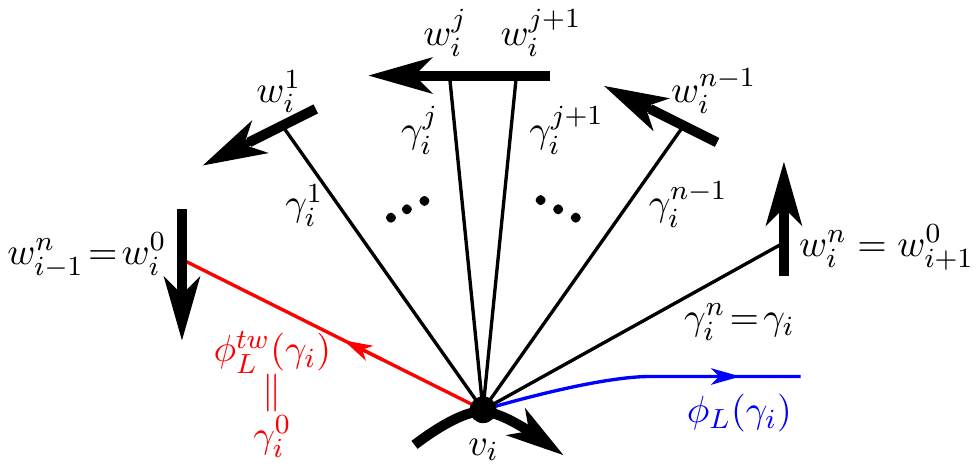} 
\caption{Interpolating sequence of arcs and their endpoints} 
\label{figA}
\end{figure}

We define a transverse overtwisted disk $D$ on which $v_i$ (resp. $w_i^j$) serves as a negative (resp. positive) elliptic point in the open book foliation of $D$. 

For the page parameter $t\in [0,1]$, let $b_{i}(t)$ denote a $b$-arc in the page $S_t$ that connects the negative elliptic point $v_i$ and positive elliptic point $w_i^j$ for some 
$j\in\{0,\ldots,n\}$ depending on $t$. 
Let $a_i^j(t)$ denote an a-arc in the page $S_t$ emanating from the positive elliptic point $w_i^j$ when $w_i^j$ is not taken by the b-arc $b_i(t)$.
(If $w_i^j$ is already taken by $b_i(t)$ then $a_i^j(t)$ does not exist.) 

Here is a movie presentation of $D$, that is a family of slices $\{\ D\cap S_t\ | \ t\in [0,1]\ \}$ such that 
$$
D\cap S_t =\left\{ b_i(t),\ a_i^j(t) \mid i=1,\ldots,N \mbox{ and }  j=1,\ldots, n\ \right\}. 
$$

On the page $S_0$ we define the b-arcs $b_i(0)$ and a-arcs $a_i^j(0)$ as follows: 
For $i=1,\ldots,N$ and $j=1,\ldots,n-1$ let $b_i(0):=\phi_L(\gamma_i^n)$ connecting the negative elliptic point $v_i$ and  the positive elliptic point $w^{n}_i$ and let  
$a_i^j(0)$ be an a-arc emanating from the positive elliptic point $w^{j}_i$. We obtain 
\[ D\cap S_0= \phi_L(\Gamma) \cup \left\{ a_i^j(0) \mid i=1,\ldots,N \mbox{ and } j=1,\ldots,n-1\right\}.\]
See the left sketch in Figure~\ref{figAA}.
\begin{figure}[htbp]
\includegraphics*[width=120mm,bb= 121 598 489 725]{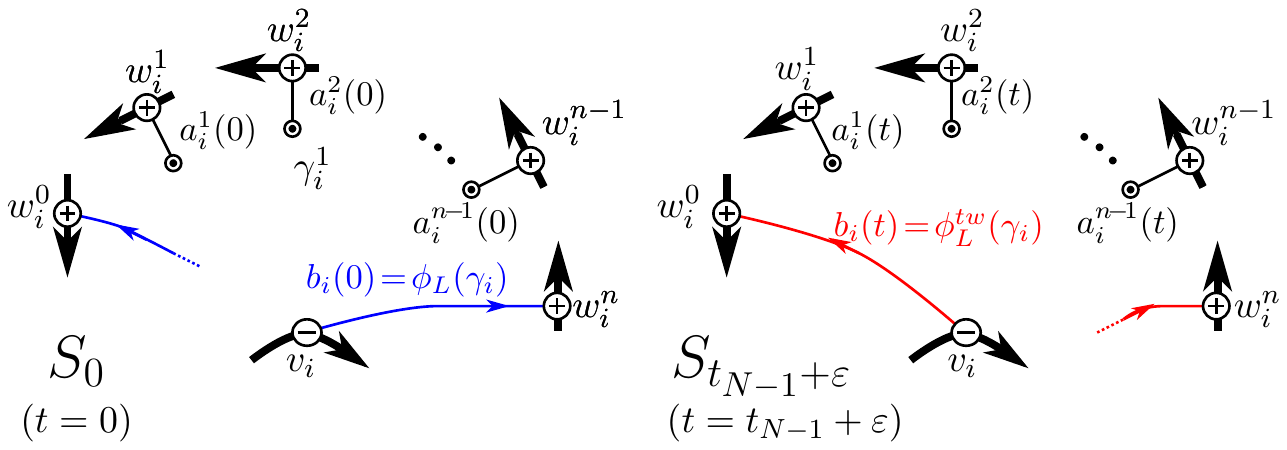}
\caption{a-arcs and b-arcs in the pages $S_0$ and $S_{t_{N-1}+\epsilon}$.} 
\label{figAA}
\end{figure}

The remaining construction consists of two steps:

(Step 1): 
Let $0<t_1<\cdots<t_{N-1}<1$ and $0<\varepsilon\ll 1$. 
For each $i=1,\ldots,N-1$ we introduce a negative hyperbolic point $h^{-}_{i}$ in the page $S_{t_i}$ whose describing arc connects $b_i(t_i-\varepsilon)$ and $b_{N}(t_i-\varepsilon)$
and is contained in $R(\Gamma, \phi_L(\Gamma))$ (see Figure \ref{figB} for $i=1, 2$). 
Therefore, $b_i(t_i-\varepsilon)=\phi(\gamma_i)$ and $b_i(t_{i}+\varepsilon)=\phi^{tw}(\gamma_i)$. 
For $i'\neq i$ we define $b_{i'}(t_i-\varepsilon):=b_{i'}(t_i+\varepsilon)$. As for a-arcs $a_i^j(t)$ where $i=1,\dots,N$, $j=1,\dots,n-1$ and $t\in[0, t_{N-1}+\varepsilon]$ no change is made during the 1st step.

Thus, after the 1st step 
(see the right sketch in Figure~\ref{figAA}) we have $(b_{1}(t),\ldots,b_{N}(t))=\phi_L^{tw}(\Gamma)=\Gamma_0$.
In other words we have at $t=\epsilon+t_{N-1}$ 
$$D\cap S_{\epsilon+t_{N-1}}= \phi_L^{tw}(\Gamma) \cup \left\{ a_i^j(\epsilon+ t_{N-1}) \mid i=1,\ldots,N \mbox{ and } j=1,\ldots,n-1\right\}.$$
\begin{figure}[htbp]
\includegraphics*[width=120mm,bb= 127 588 483 713]{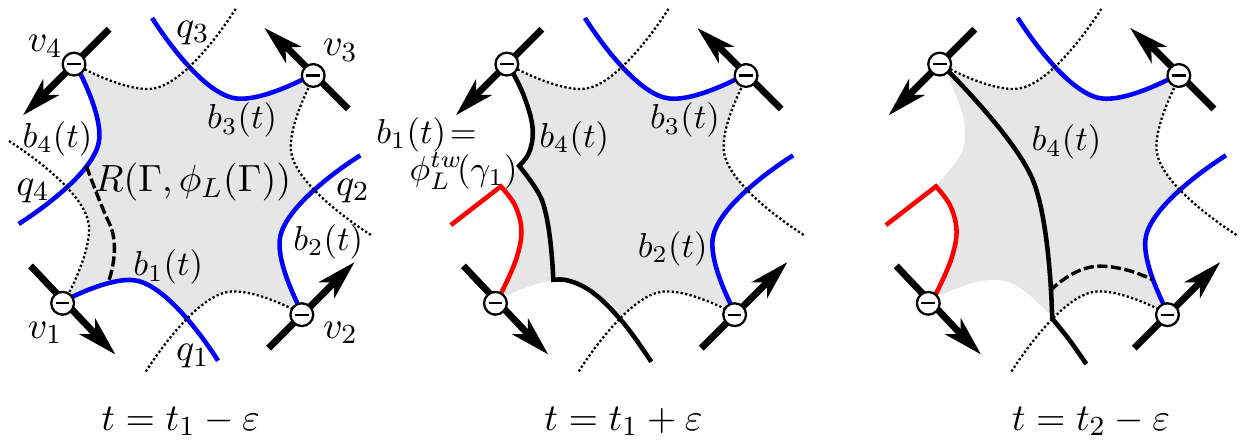}
\caption{Movie presentation for Step 1. Dashed arcs are describing arcs for $h_1^-$ and $h_2^-$.}
\label{figB}
\end{figure} 

The open book foliation of $D$ described in Step 1 is depicted in Figure~\ref{figC} where $N=4$. 
Since $R=R(\Gamma,\phi(\Gamma))$ contains $k$ puncture points, the portion of overtwisted disk $D$ constructed in this step intersects $L$ at $k$ times.

\begin{figure}[htbp]
\includegraphics*[width=60mm,bb= 211 546 399 726]{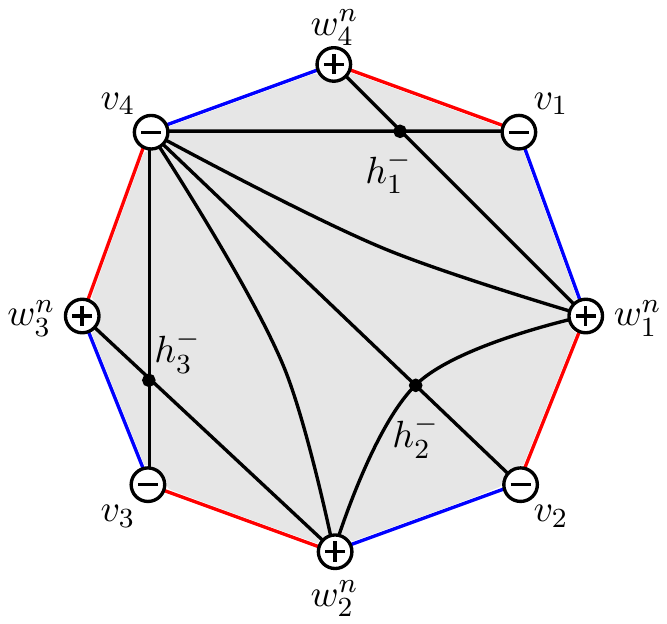}
\caption{The open book foliation for Step 1. (Fan shaped regions swept by a-arcs during Step 1 are not included here, but they are included in Figure~\ref{figD2} below.) If $k>0$ then $k$ puncture points should be added to it.} 
\label{figC}
\end{figure} 

(Step 2): 
Let $t_{N-1}<s_0<s_1<\cdots<s_{n-1}<1$. 
For each $j=0,\ldots, n-1$, we change the arc system $\Gamma_j=\left(b_{1}(s_j-\varepsilon),\ldots, b_{N}(s_j-\varepsilon)\right)$ into $\Gamma_{j+1}=(b_{1}(s_j+\varepsilon),\ldots, b_{N}(s_j+\varepsilon))$ by introducing $N$ positive hyperbolic points simultaneously whose describing arcs are parallel to $\gamma^{j+1}_{i}$ for $i=1,\dots,N$ and connecting $b_{i}(s_j-\varepsilon)$ and $a^{j+1}_{i}(s_j-\varepsilon)$, see Figure~\ref{figD1} for a movie presentation.
The open book foliation of the $t$-interval $[s_j-\varepsilon,s_{j+1}-\varepsilon]$ is the shaded region in the right sketch of Figure~\ref{figD2}.
\begin{figure}[htbp]
\includegraphics*[width=80mm, bb=184 592 435 717]{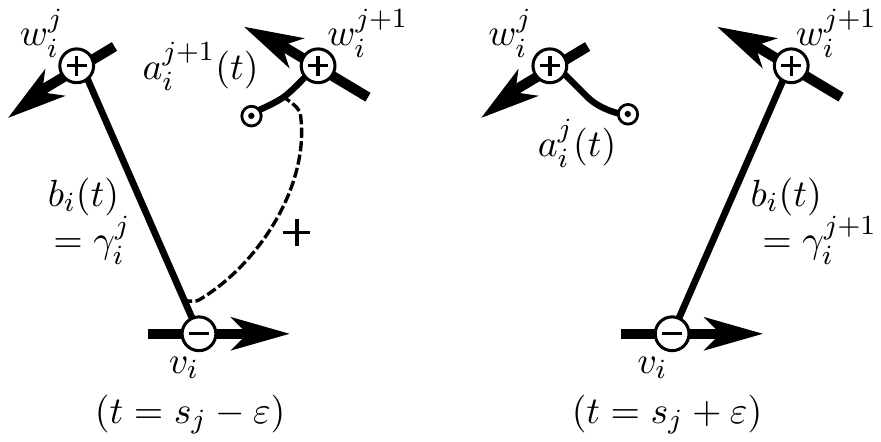}
\caption{Movie presentation for Step 2.} 
\label{figD1}
\end{figure} 
Since $\Gamma_j \disj \Gamma_{j+1}$ these describing arcs can be taken to be pairwise disjoint that enables us to introduce $N$ positive hyperbolic points simultaneously. 
No changes are made to the rest of the $a$-arcs (the ones from $w_i^{j'}$ with $j'\neq j, j+1$.)

In short, before and after $t=s_j$, for every $i=1,\ldots,N$ the b-arc for $v_i$ changes from $b_{i}(s_j-\varepsilon)=\gamma_{i}^j$ to $b_{i}(s_j+\varepsilon)=\gamma_{i}^{j+1}$ and we obtain: 
$$
D\cap S_{s_j+\epsilon} = \Gamma_{j+1} \cup \left\{ \mbox{ a-arcs } \right\}
$$

Eventually, at $t=1$ we obtain $(b_1(1),\ldots, b_N(1))=\Gamma_k=\Gamma$. 
Therefore the intersection of $D$ with the page $S_1$ consists of b-arcs $(b_1(1), \ldots, b_N(1))=\Gamma$ and a-arcs emanating from the positive elliptic points $w^{j}_i$ where $i=1,\ldots,N$ and $j=1,\ldots,n-1$. By the distinguished monodromy $\phi_L$ these arcs are mapped to the b-arcs $\phi_L(\Gamma)$ and a-arcs on the page $S_0$ to give a transverse overtwisted disk which has $N$ negative elliptic points.

During Step 2, since $P$ and leaves of the open book foliation of $D$ do not intersect, the portion of the overtwisted disk $D$ constructed in Step 2 is disjoint from $L$. Thus the resulting overtwisted disk $D$ intersects $L$ in total at $k$ times.

The left sketch in Figure \ref{figD2} depicts the whole open book foliation of $D$. 
\begin{figure}[htbp]
\includegraphics*[width=120mm, bb=127 549 483 714]{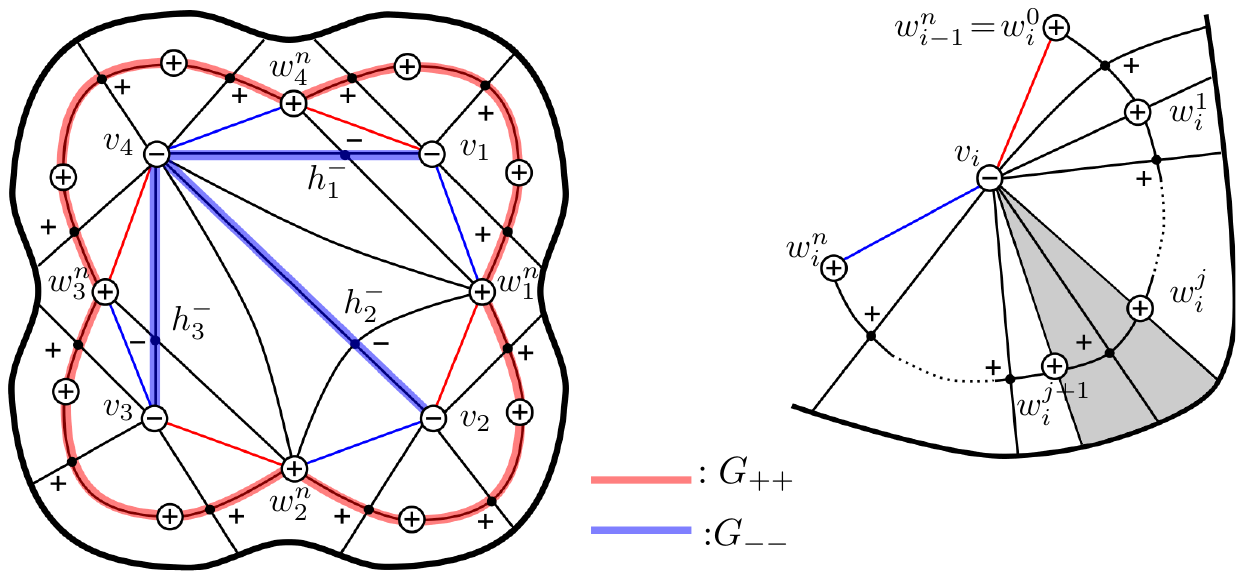}
\caption{The entire open book foliation of $D$ ($N=4$ case) and Step 2 construction.} 
\label{figD2}
\end{figure}
\end{proof}

\section{Application to depth of transverse links}
\label{section5}

Recall the depth of a transverse link defined by Baker and Onaran \cite{BO}: 

\begin{definition}\cite{BO}
Let $\mathcal T$ be a transverse link in an overtwisted contact 3-manifold $(M, \xi)$. 
The depth $\dep(\mathcal T;M)$ of $\mathcal T$ is defined by:
\[ \dep(\mathcal T;M) = \min \{\#(\mathcal T \cap D)  \: | \: D \text{ is an overtwisted disk in } (M,\xi) \} 
\]
\end{definition}
The depth measures non-looseness of transverse links. In particular, $\dep(\mathcal T;M)=0$ if and only if $\mathcal T$ is loose. 

In \cite{IK2} we have defined the \emph{overtwisted complexity} $n(S,\phi)$ of an open book $(S,\phi)$.
\begin{definition}\cite{IK2}. 
Let 
\[
n(S, \phi) = \min \left\{
e_-(D) \ \middle| \ D \subset M_{(S, \phi)}: \mbox{transverse overtwisted disk} \right\}
\]
where $e_-(D)$ is the number of negative elliptic points in the open book foliation of $D$. 
It is an invariant of the open book and we call it the {\em overtwisted complexity} of $(S, \phi)$. Recall that the binding $B_{(S, \phi)}$ is a transverse link. 
In \cite[Theorem 3]{IK-cover} it is shown that 
\begin{equation}\label{n=depth}
n(S, \phi)=\dep(B_{(S,\phi)};M_{(S, \phi)}).
\end{equation}  
\end{definition}

Theorem~\ref{theorem:non-Nrv-OT} gives upper bounds of the depth of closed braids and binding.

\begin{corollary}\label{cor:depth}
Let $(S,\phi)$ be an open book supporting an overtwisted contact structure, $L$ be a closed braid in $M_{(S,\phi)}$, and $B_{(S,\phi)}$ be the binding.
\begin{itemize}
\item[(a)] $\dep(B_{(S,\phi)};M_{(S, \phi)})\leq  \min\{N\: | \: \phi \mbox{ is } N\mbox{-twist left-veering}\}$.
\item[(b)] $\dep(B_{(S, \phi)};M_{(S, \phi)}\setminus L) \leq  \min\{N\: | \: \phi_L \mbox{ is } (N,0)\mbox{-twist left-veering} \}$
\item[(c)] $\dep(L;M_{(S, \phi)})\leq \min\{k\: | \: \phi_L \mbox{ is } (N,k)\mbox{-twist left-veering for some } N \}.$
\item[(d)] $\dep(L\cup B_{(S,\phi)};M_{(S, \phi)})\leq \min\{N+k\: | \: \phi_L \mbox{ is } (N,k)\mbox{-twist left-veering}\}.$ 
\end{itemize}
\end{corollary}

\begin{question}
\label{conj:depth}
Can the inequalities (a), \dots, (d) in Corollary~\ref{cor:depth} be equalities?

\end{question}

\begin{remark}
In some cases $\dep(L\cup B_{(S,\phi)};M_{(S, \phi)}) = \dep(B_{(S,\phi)};M_{(S, \phi)}\setminus L)$ but in general 
\begin{equation}\label{eq:depth-difference}
\dep(B_{(S,\phi)};M_{(S, \phi)}\setminus L)-\dep(L\cup B_{(S,\phi)};M_{(S, \phi)})\geq 0
\end{equation}
and the difference can be arbitrary large. 
\end{remark}

The next lemma concerns the equality of (\ref{eq:depth-difference}).

\begin{lemma}
\label{lemma:depthcomparison}
Suppose that $N \leq \dep(B_{(S,\phi)};M_{(S, \phi)})$. We have $\dep(L\cup B_{(S,\phi)};M_{(S, \phi)})=N$ if and only if $\dep(B_{(S,\phi)};M_{(S, \phi)}\setminus L)=N$.
In particular, $\dep(L\cup B_{(S,\phi)};M_{(S, \phi)})=1$ if and only if $\dep(B_{(S,\phi)};M_{(S, \phi)}\setminus L)=1$.
\end{lemma}

\begin{proof}
Let $B:=B_{(S, \phi)}$ and $M:=M_{(S, \phi)}$. 

($\Leftarrow$): 
If $\dep(B; M\setminus L)=N$ then 
by 
$$
N\leq\dep(B;M) \leq  \dep(L\cup B;M) \leq \dep(B;M\setminus L)=N
$$
we get
$\dep(L\cup B;M)=N$. 

($\Rightarrow$): 
Conversely, assume that $\dep(L\cup B;M)=N$. There exists an overtwisted disk $D$ that intersects $L\cup B$ at $N$ points.
By 
$$
\dep(B; M) \leq \dep(L\cup B ; M) = N \leq \dep(B;M)
$$
all the $N$ intersection points belong to $B$ and $D \subset M\setminus L$.
Thus $\dep(B ; M\setminus L) \leq |B \cap D|=N$. 
By (\ref{eq:depth-difference}) we get
$$
N= \dep(L\cup B ; M) \leq \dep(B ; M\setminus L) \leq N,$$
done.
\end{proof}

The answers to Question~\ref{conj:depth} for (a),(b) and (d) are affirmative for the depth $1$ case:

\begin{proposition}\cite[Corollary 1]{IK-cover} \cite[Theorem 5.5]{IK-qveer} 
\label{prop:depth1}
Suppose that $\xi_{(S, \phi)}$ is overtwisted. 
Let $L$ be a closed braid in $M_{(S, \phi)}.$
\begin{itemize}
\item
$\dep(B_{(S,\phi)};M_{(S, \phi)})=1$ if and only if $\phi$ is non-right-veering $($i.e. $1$-twist left-veering$).$ 
\item
$\dep(L\cup B_{(S,\phi)};M_{(S, \phi)})=\dep(B_{(S,\phi)};M_{(S, \phi)} \setminus L) = 1$ if and only if $\phi_L$ is non-quasi-right-veering $($i.e., $(1,0)$-twist left-veering$).$ 
\end{itemize}
\end{proposition}

Here is another fact supporting the affirmative answer to  Question~\ref{conj:depth}. We prove equalities (a) and (b) for the depth 2 case.

\begin{theorem}\label{prop:depth2}
Suppose that $\xi_{(S, \phi)}$ is overtwisted. Let $L$ be a closed braid in $M_{(S, \phi)}.$
\begin{itemize}
\item
$\dep( B_{(S,\phi)};M_{(S, \phi)})= 2$ if and only if $\phi$ is right-veering and 2-twist left-veering.
\item
$\dep(B_{(S,\phi)};M_{(S, \phi)}\setminus L)= 2$ if and only if $\phi_L$ is quasi-right-veering and $(2,0)$-twist left-veering.
\end{itemize}
\end{theorem}

\begin{proof}
We prove the second statement as the first statement is a special case of the second. 

($\Rightarrow$)
Assume that $\dep(B_{(S,\phi)};M_{(S, \phi)} \setminus L)= 2$.
By Proposition \ref{prop:depth1} we know that $\phi_L$ is quasi-right-veering. Thus, it is enough to show that $\phi_L$ is $(2,0)$-twist left-veering. 

There exists a transverse overtwisted disk $D \subset M_{(S, \phi)}\setminus L$ with exactly two negative elliptic points, which we call $v_1$ and $v_2$. 
For $t\in[0,1]$ and $i=1,2$, let $b_i(t)$ denote the b-arc (if exists) in $D\cap S_t$ that ends at $v_i$.

Since the graph $G_{-,-}$ of any transverse overtwisted disk is a tree, $D$ has exactly one negative hyperbolic point, $h_-$, that is connected to $v_1$ and $v_2$ by a singular leaf.
Let $\{S_{t_n}\ |\ n=1,\cdots, k\}$ where $0<t_1<t_2<\cdots <t_k <1$ be the set of singular pages. For a transverse overtwisted disk with two negative elliptic points, $k\geq 4$. 
Assume that the unique negative hyperbolic point $h_-$ lies in $S_{t_1}$ and each of the pages $S_{t_2}, \ldots, S_{t_k}$ contains one positive hyperbolic point.

For $t\neq t_1,\ldots,t_k$ let 
$\Gamma_t:=(b_1(t), b_2(t))$ be a $2$-arc system with the base $\mathcal B=\{v_1, v_2\}$ consisting of the b-arcs of $D$ in the page $S_t$ (strictly speaking it should be $\Gamma_t:=(p_t(b_1(t)), p_t(b_2(t)))$ where $p_t : S_t \to S$ is the canonical diffeomorphism defined in Section~\ref{section2.1}).
Then we have $\Gamma_0=\phi_L(\Gamma_1)$. Since each positive hyperbolic point veers a b-arc to the right with respect to the negative elliptic point $v_i$ and the resulting arc is disjoint from the original one, we have
\[
\Gamma_{t_i-\epsilon} \disj \Gamma_{t_i + \epsilon} = \Gamma_{t_{i+1} - \epsilon}
\]
therefore, 
\[ \Gamma_{t_1+\epsilon} \disj \Gamma_{t_2+\epsilon} \disj \cdots \disj \Gamma_{t_k+\epsilon}=\Gamma_1.
\] 

We may assume that the end points of the describing arc for the hyperbolic point $h_-$ are very close to $v_1$ and $v_2$ (see the left sketch in Figure~\ref{fig:depth2}).  
The boundary based 4-gon $R(\Gamma,\phi_L(\Gamma))$ exists in a neighborhood of the describing arc (the shaded region in the 4th sketch of Figure~\ref{fig:depth2}).  
This shows that $\phi_L^{tw}(\Gamma_1)=\Gamma_{t_1+\epsilon}$ hence 
\[\phi_L^{tw}(\Gamma_1)= \Gamma_{t_1+\epsilon}\ll_{\sf right} \Gamma_1.\] 
The 4-gon $R(\Gamma,\phi_L(\Gamma))$ is not punctured since $L$ is disjoint from the overtwisted disk $D$;
i.e., $\phi_L$ is $(2,0)$-twist left-veering. 
\begin{figure}[htbp]
\includegraphics*[width=130mm,bb=  99 567 508 719]{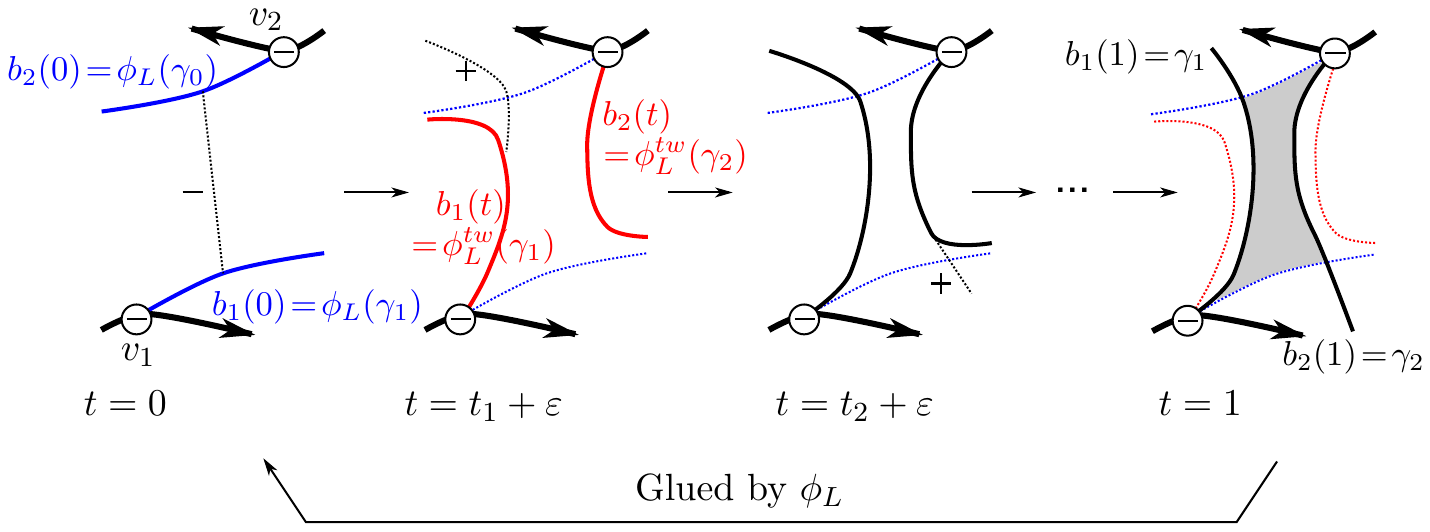}
\caption{Movie presentation of a transverse overtwisted disk with two negative elliptic points.} 
\label{fig:depth2}
\end{figure} 

($\Leftarrow$)
This implication follows by Corollary~\ref{cor:depth} and Proposition~\ref{prop:depth1}, 
\end{proof}
 
\begin{corollary}
If $\phi$ is right-veering and $\xi_{(S, \phi)}$ is overtwisted then  
$\dep(L\cup B_{(S, \phi)};M_{(S, \phi)})=2$ if and only if $\phi_L$ is quasi-right-veering and $(2,0)$-twist left-veering. 
\end{corollary}
 
\begin{proof}
By Proposition~\ref{prop:depth1}, the map $\phi$ is right-veering if and only if $2 \leq \dep(B_{(S, \phi)}, M_{(S, \phi)})$. 
By Lemma~\ref{lemma:depthcomparison}, $\dep(L\cup B_{(S, \phi)};M_{(S, \phi)})=2$ if and only if $\dep(B_{(S, \phi)}, M_{(S, \phi)}\setminus L)=2,$ which is equivalent to $\phi_L$ is quasi-right-veering and $(2,0)$-twist left-veering by  Theorem~\ref{prop:depth2}.
\end{proof}

\section{Connection to Wand's inconsistency}
\label{section6}

In this section we discuss the relation between $(N,0)$-twist left-veering and Wand's inconsistency. 
We begin with definitions of overtwisted region and inconsistent mapping class that Wand introduced in \cite[Definition 2.4]{Wand1}, where the puncture set $P$ is empty.

\begin{definition}[Wand]\label{def:OT-region}
Let $\phi\in\MCG(S)$ and $\Gamma\in\mathcal A_{\mathcal B}(S)$ an $N$-arc system with $N\geq 1$.  
An {\em overtwisted region}  in the augmented open book $(S, \phi, \Gamma)$ is a  $2N$-gon disk $A$ embedded in $S$ (when $N=1$ relaxing the condition that $\Gamma$ and $\phi(\Gamma)$ intersect efficiently, $A$ can be a bigon)
and $\partial A \subset (\Gamma \cup \phi(\Gamma))$, such that 
\begin{enumerate}
\item 
The orientation of $\partial A$ 
\begin{itemize}
\item
agrees with that of $\phi(\Gamma)$ and disagrees with that of $\Gamma$, or 
\item
disagrees with that of $\phi(\Gamma)$ and agrees with that of $\Gamma$.
\end{itemize}

\item
Each point of $\Gamma\cap\phi(\Gamma)\cap\Int(S)$ is a corner of $A$. Corners of $A$ alternate between points in $\partial \Gamma= \mathcal B \cup \Gamma(1)$ and points in $\Gamma\cap\phi(\Gamma)\cap\Int(S)$. 

\item
$A$ is the unique such disk.
\end{enumerate}
\end{definition}

As the name suggests, Wand showed that:
\begin{proposition}
If $(S, \phi, \Gamma)$ has an overtwisted region then $(S,\phi)$ supports an overtwisted contact structure. 
\end{proposition}

\begin{remark}
Wand's proof of the above statement does not immediately generalize to surfaces with $P\neq \emptyset$. His construction of an overtwisted disk $D$ uses $\Gamma\times[0,1]$ as part of $D$, but in general $\Gamma\times[0,1]$ may intersect the transverse link $L$ that corresponds to $P$. 

Moreover, even the definition of an overtwisted region does not immediately extend to $P\neq \emptyset$ case. For example, if $\phi(\gamma) \prec_\ri \gamma$ and they bound a punctured bigon and $F(\phi)=id \in \MCG(S)$ under the forgetful map $F:\MCG(S, P)\to \MCG(S)$ then one can introduce a bigon at the base point $\gamma(0)$ satisfying the conditions (1)-(3). However the contact structure supported by $(S, \phi)$ is tight.

\end{remark}

Now we recall the definition of inconsistency. 

\begin{definition}[Wand]
A class $\phi \in\MCG(S)$ is {\em inconsistent} if there is some arc system $\Gamma$ in $S$ and 
a stabilization $(S',\phi')$ of $(S,\phi)$ such that $(S', \phi', \iota(\Gamma))$ has an overtwisted region, where $\iota: S \to S'$ is the inclusion map (for simplicity $\iota(\Gamma)$ is denoted by $\Gamma$ in the following). Otherwise, $\phi$ is {\em consistent}.  
\end{definition}

Inconsistency is a central concept in Wand's work due to the following result: 
\begin{theorem}\cite[Theorem 1.1]{Wand1}\label{wand'sThm1.1}
\label{theorem:Wand}
$(S,\phi)$ supports an overtwisted contact structure if and only if $\phi$ is inconsistent. 
\end{theorem}

We show that existence of overtwisted region can be understood as a special case of $N$-twist left-veering.

\begin{theorem}\label{thm:R=A}
Let $\Gamma$ be an $N$-arc system with $N\geq 2$ such that the boundary based region $R(\Gamma,\phi(\Gamma))$ exists.  
Then  $R(\Gamma,\phi(\Gamma))$ is an overtwisted region if and only if
$\phi^{tw}(\Gamma) \ll_\ri \Gamma$, $\Int(\phi^{tw}(\Gamma))\cap\Int(\Gamma)=\emptyset$, and $R(\Gamma,\phi(\Gamma))$ is embedded. 
\end{theorem}

\begin{proof}
($\Leftarrow$) 
Assume that $\phi^{tw}(\Gamma) \ll_\ri \Gamma$, $\Int(\phi^{tw}(\Gamma))\cap\Int(\Gamma)=\emptyset$, and $\Int(R(\Gamma,\phi(\Gamma)) \cap \Gamma=\emptyset$.

Condition (1) of Definition  \ref{def:OT-region} is clearly satisfied. 

Since $\Int(R(\Gamma,\phi(\Gamma)) \cap \Gamma=\emptyset$ we  have 
$\partial R(\Gamma,\phi(\Gamma)) \cap \Int(\phi(\Gamma)) \cap \Int(\Gamma) =\{q_1,\ldots,q_N\}$.
Then 
\begin{eqnarray*}
\Gamma \cap \phi(\Gamma) \cap \Int(S) &=& [\Int(\phi^{tw}(\Gamma))\cap \Int(\Gamma)]  \cup \left[\partial R(\Gamma,\phi(\Gamma)) \cap \Int(\phi(\Gamma)) \cap \Int(\Gamma)\right]\\
&=&
\{q_1,\ldots,q_N\}
\end{eqnarray*}
i.e., Condition (2) follows.

Assume to the contrary that Condition (3) does not hold.
By (2), this happens only if $R(\underline{\Gamma},\phi(\underline{\Gamma}))\neq \emptyset$ where $\underline{\Gamma}$ denotes the $\Gamma$ with reversed orientation (see Figure \ref{fig:WandOTregion} (a)).
\begin{figure}[htbp]
\includegraphics*[width=80mm, bb=177 602 427 716]{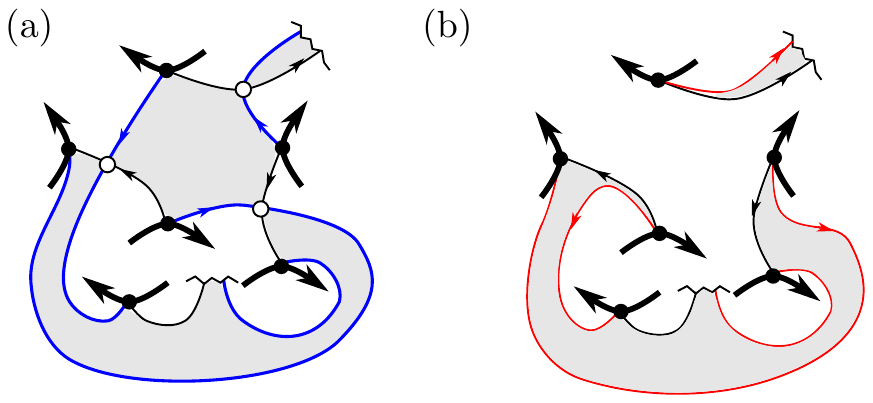}
\caption{(a) Regions that satisfy Properties (1) and (2) but (3). 
(b) The boundary based region $R(\phi^{tw}(\Gamma), \Gamma)$.}
\label{fig:WandOTregion}
\end{figure} 
We note that $R(\phi^{tw}(\Gamma),\Gamma)$ is non-empty as it contains $R(\underline{\Gamma},\phi(\underline{\Gamma}))$, see Figure \ref{fig:WandOTregion} (b).
By Proposition \ref{prop:Robstruction} this implies $\phi^{tw}(\Gamma) \not\ll_{\sf right}\Gamma$ which contradicts the assumption.

($\Rightarrow$) 
To prove the only if direction, assume that $R(\Gamma,\phi(\Gamma))$ satisfies (1)--(3) of Definition \ref{def:OT-region}.
Then clearly $\Int(\Gamma) \cap \Int(\phi^{tw}(\Gamma))=\emptyset$ and $\Int(R(\Gamma,\phi(\Gamma)) \cap \Gamma=\emptyset$ are satisfied. It is left to show $\phi^{tw}(\Gamma) \ll_{\ri} \Gamma$.

Let $\Gamma=(\gamma_1,\ldots,\gamma_N)$. Define $\Gamma_1=(\gamma^1_1,\ldots,\gamma^1_N)$ to be an arc system obtained from $\phi^{tw}(\Gamma)$ by shifting the endpoints $\phi^{tw}(\Gamma)(1)$ slightly to the left (with respect to the orientation of the boundary $\partial S$). Then $\phi^{tw}(\Gamma) \disj \Gamma_1 \prec_{\ri} \Gamma$. 
For $i=1,\ldots,N$ let $R'_i$ be the connected component of $S \setminus (\Gamma \cup \Gamma_1)$ which lies between $\gamma^1_i$ and $\gamma_i$. We note that $R'_i$ cannot be a
boundary based region because otherwise Condition (3) would be violated. 

If $\partial R'_i$ contains a boundary component of $S$, let $\gamma_i^2$ be a properly embedded arc in $R'_{i}$ connecting $v_i$ and the boundary component of $S$ (see the left sketch in Figure \ref{fig:WandOTregion2}). We have $\gamma_i^1\disj \gamma_i^2 \disj \gamma_i$. Let $\gamma_i^3$ be an arc obtained from $\gamma_i^2$ whose endpoint $\gamma_i^2(1)$ is slightly moved so that $\gamma_i^1\disj \gamma_i^2 \disj \gamma_i^3 \disj \gamma_i$.

If $\partial R'_i$ does not contain any boundary components of $S$ then $R'_i$ has genus $\geq 1$. We construct arcs $\gamma_i^2$ and $\gamma_i^3$ from simple closed curves in $R'_i$ based at $v_i$ that go around a handle (see the right sketch in Figure \ref{fig:WandOTregion2}) and whose endpoints $\gamma_i^2(1)$ and $\gamma_i^3(1)$ are slightly moved to realize $\gamma_i^1\prec^{\sf disj} \gamma_i^2 \prec^{\sf disj} \gamma_i^3 \prec^{\sf disj} \gamma_i$.

\begin{figure}[htbp]
\includegraphics*[width=90mm, bb=184 583 440 713]{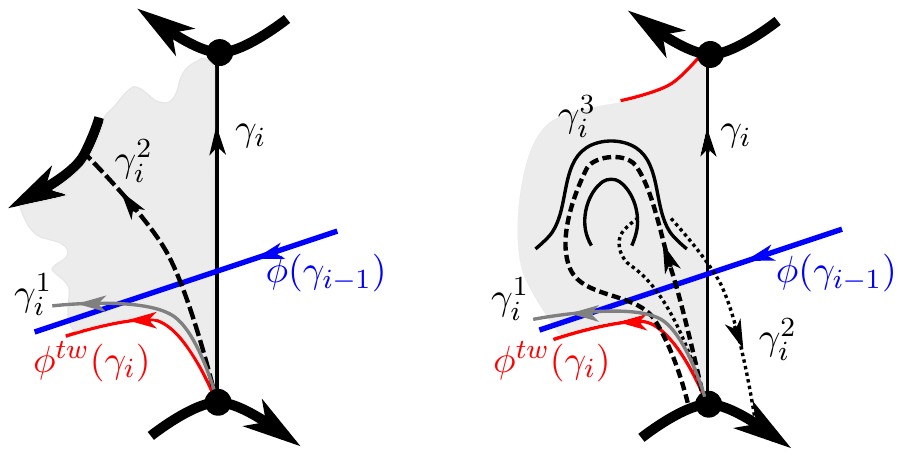}
\caption{Construction of arc systems $\phi^{tw}(\Gamma) \prec^{\sf disj}\Gamma_1  \prec^{\sf disj}\Gamma_2  \disj \Gamma_3 \prec^{\sf disj} \Gamma$.}
\label{fig:WandOTregion2}
\end{figure} 

We take such arcs so that $\Gamma_2=(\gamma^2_1,\ldots,\gamma_N^2)$ and $\Gamma_3 = (\gamma^3_1,\ldots,\gamma_N^3)$ are $N$-arc systems with 
\[ \phi^{tw}(\Gamma) \prec^{\sf disj}\Gamma_1  \prec^{\sf disj}\Gamma_2  \disj \Gamma_3 \prec^{\sf disj} \Gamma,\]
hence, $\phi^{tw}(\Gamma) \ll_{\ri} \Gamma$.
\end{proof}

When $N=1$ a parallel statement to Theorem~\ref{thm:R=A} holds: 

\begin{proposition}\label{prop:R=A}
Let $\gamma$ be an arc (i.e., a 1-arc system) in $S$.  
The augmented open book  $(S, \phi, \gamma)$ has an overtwisted region if and only if $\phi^{tw}(\gamma) \ll_\ri \gamma$ and that $\Int(\phi^{tw}(\gamma)) \cap \Int(\gamma) =\emptyset$. 
\end{proposition}

\begin{proof}
Note that $\phi(\gamma)=\phi^{tw}(\gamma)$ since $\gamma$ is a 1-arc system. 

($\Rightarrow$) 
Suppose that $(S, \phi, \gamma)$ has a bigon overtwisted region. 
Reversing the orientation of $\gamma$ if necessary, we may further assume that the bigon is at the base point  $\gamma(0)$. 
By the uniqueness property (3) of Definition~\ref{def:OT-region} we know that $\phi\neq id$. 
After removing the bigon formed by $\gamma$ and $\phi(\gamma)$ we see that $\phi(\gamma)\prec_\ri \gamma$ and $\phi(\gamma)\cap \gamma = \partial \gamma$; that means $\gamma$ and $\phi(\gamma)$ are almost disjoint. 
Since $\phi\neq id$ the component of $S\setminus (\phi(\gamma) \cup \gamma)$ that lies between $\phi(\gamma)$ and $\gamma$ at the base point $\gamma(0)$ is not a disk. 
Applying Honda, Kazez and Mati\'c's algorithm \cite{HKM} we can find a sequence of arcs with 
$$\phi(\gamma) \disj \gamma_1 \disj\dots\disj \gamma_l \disj \gamma,$$ thus $\phi^{tw}(\gamma)=\phi(\gamma) \ll_\ri \gamma$. 

($\Leftarrow$) 
Assume that $\phi^{tw}(\gamma) \ll_{\ri} \gamma$ and that $\Int(\phi^{tw}(\gamma)) \cap \Int(\gamma)=\emptyset$. Since $\phi^{tw}(\gamma)=
\phi(\gamma)$ we have $\Int(\phi(\gamma)) \cap \Int(\gamma)=\emptyset$ so by isotopy, $\phi(\gamma)$ and $\gamma$ form a bigon at the base point which is an overtwisted region for $(S, \phi, \gamma)$. 
\end{proof} 

By Theorems~\ref{theorem:non-Nrv-OT}, \ref{wand'sThm1.1} and \ref{thm:R=A}, and Proposition~\ref{prop:R=A} we obtain the following corollary.

\begin{corollary}
\label{cor:stably-LV}
An open book $(S,\phi)$ supports an overtwisted contact structure if and only if a stabilization of $(S,\phi)$ is $N$-twist left-veering for some $N$.
\end{corollary}

\section{Variation of twist-left-veering and virtual looseness}\label{section7}

One motivation to introduce the ordering $\ll_{\ri}$ comes from an observation that non-right-veering closed braids are not necessarily loose. 
On the other hand, in \cite[Corollary 5.7]{IK-branch} we showed that non-right-veering closed braids are \emph{virtually loose}; that is, some finite cover of its complement is overtwisted. 
In this section we will generalize the result in \cite{IK-branch} from arcs to $N$-arc systems. 
We do this by introducing a variation $\ll_\ri^{\partial+P}$ of the ordering $\ll_\ri$.

We begin with reviewing the standard branched cyclic coverings (studied in \cite{IK-branch}) and then discuss how twist-left-veering can be related to virtual looseness.

Let $d>1$. 
For an $m$-component transverse link $\mathcal{T}=\mathcal{T}_1\cup \cdots \cup \mathcal{T}_m$ in a contact 3-manifold $(M,\xi)$.  
Let
\[ e_d = p_d \circ H: \pi_{1}(M\setminus \mathcal{T}) \rightarrow \Z \slash d\Z \]
be a homomorphism obtained by composing the Hurewicz homomorphism 
$$H:\pi_{1}(M\setminus \mathcal{T}) \rightarrow H_1(M\setminus \mathcal{T};\Z)$$ 
and 
$$p_d: H_1(M\setminus \mathcal{T};\Z) \cong H_{1}(M;\Z) \oplus \bigoplus_{i=1}^{m} \Z[\mu_i] \rightarrow \Z \slash d\Z$$ 
that is defined by $p_d(x)=0$ for $x \in H_1(M;\Z)$ and $p_d([\mu_i])=1$ for all $i=1,\dots,m$, where $[\mu_{i}] \in  H_1(M\setminus \mathcal{T};\Z)$ be the homology class represented by a meridian $\mu_i$ of the $i$-th component $\mathcal{T}_i$.
The \emph{standard} $d$-fold cyclic branched covering of $\mathcal{T}$ is a contact branched covering \cite{Geiges, ON}. 
$$\pi:(\widetilde{M},\widetilde{\xi})\to (M, \xi)$$ such that the restriction $\pi: \widetilde{M} \setminus \widetilde{\mathcal{T}} \rightarrow M\setminus \mathcal{T}$ is a usual covering that corresponds to $\textrm{Ker}\,e_d$. 

Assume that $(M,\xi)$ is supported by an open book $(S,\phi)$ and $\mathcal{T}$ is represented by a closed $n$-braid $L$ in $(S,\phi)$ with the distinguished monodromy  $\phi_L \in \MCG(S, P)$ where $P$ is a set of $n$ interior points of $S$. 
Let 
\[ e'_d=p'_d \circ H': \pi_{1}(S \setminus P) \rightarrow \Z \slash d\Z \]
be a homomorphism defined by composing the Hurewicz homomorphism $$H':\pi_{1}(S\setminus P) \rightarrow H_1(S\setminus P;\Z)$$ and $$p'_d: H_1(S\setminus P;\Z)\cong H_{1}(S;\Z) \oplus \bigoplus_{i=1}^{n} \Z[c_i] \rightarrow \Z \slash d\Z$$ defined by $p'_d(x)=0$ for $x \in H_1(S;\Z)$ and $p'_d([c_i])=1$ for $i=1,\ldots, n$ where $[c_i]$ is the homology class represented by a loop around $p_i \in P$.

Let $\pi_{S}=\pi^{(d)}_{S}: (\widetilde{S}, \widetilde P) \rightarrow (S, P)$ be the $d$-fold cyclic branched covering that corresponds to $\textrm{Ker}\,e'_d$. We call it the \emph{standard} $d$-fold cyclic branched covering of $(S,P)$.
The distinguished monodromy $\phi_L$ always lifts to $\widetilde{\phi_L}:(\widetilde{S}, \widetilde P) \rightarrow (\widetilde{S}, \widetilde P)$ which we call the \emph{standard} lift.

\begin{lemma}
\label{lemma:virtual-loose}
In the above setting, 
if the standard lift $\widetilde{\phi_L}:(\widetilde{S}, \widetilde P) \rightarrow (\widetilde{S}, \widetilde P)$ is $(N, 0)$-twist-left-veering 
then $\mathcal T$ is virtually loose. 
\end{lemma}

\begin{proof}
The lift $\widetilde{L}:=\pi^{-1}(L)$ is a closed braid representative of $\widetilde{\mathcal{T}} = \pi^{-1}(\mathcal{T})$ and its distinguished monodromy is $\widetilde{\phi_L}$. 
If $\widetilde{\phi_L}$ is twist-left-veering with a non-punctured boundary based region then by Theorem~\ref{theorem:non-Nrv-OT} we know that $\widetilde{L}$ is loose. 
Thus $\mathcal{T}$ is virtually loose. 
\end{proof}

The above lemma motivates us to ask when the standard lift $\widetilde{\phi_L}$ becomes twist-left-veering. 
To this end, we extend the right-veering orderings to a slightly bigger set.

\begin{definition}
A \emph{$(\partial+P)$-arc} $\gamma$ is an oriented properly embedded arc in $S\setminus P$ with the starting point $\gamma(0)\in \partial S$ and the end point $\gamma(1) \in \partial S\cup P$. 

A \emph{$(\partial+P)$-arc system} $\gamma=(\gamma_1,\ldots,\gamma_N)$ is defined similarly with the set of endpoints 
$\Gamma(1)\subset \partial S \cup P$. For a set of $N$ boundary points $\B \subset \partial S$ let $\AP$ denote the set of $(\partial+P)$-arc systems that start at $\B$. 

The orderings $\prec_{\ri}$ and $\ll_{\ri}$, and the relation $\disj$ on $\A$ can be extended to $\AP$ and denoted by $\prec^{\partial+P}_{\ri}$, $\ll^{\partial+P}_{\ri}$, and $\disjp$.
\end{definition}

\begin{remark}
By the definition $\A \subset \AP$.

For $\Gamma,\Gamma' \in \A \subset \AP$,
we have 
$\Gamma \prec_{\ri} \Gamma'$ if and only if $\Gamma \prec^{\partial+P}_{\ri} \Gamma'$.
However, in general $\Gamma \ll_{\ri}^{\partial+P} \Gamma'$ does not imply $\Gamma \ll_{\ri} \Gamma'$ since  $\Gamma \disjp  \Gamma'$ does not imply $\Gamma \disj \Gamma'$. 
\end{remark}

As stated in Proposition \ref{prop:Robstruction}, when the boundary based region $R(\Gamma,\Gamma')$ is nonempty then $\Gamma \not \ll_{\ri} \Gamma'$. This is not the case for $\ll_\ri^{\partial+P}$: 

\begin{lemma}
If a boundary based region $R(\Gamma,\Gamma')$ is embedded (i.e., its interior does not intersect $\Gamma$) and contains a puncture point, then $\Gamma \ll_{\ri}^{\partial+P} \Gamma'$ (see Figure \ref{Fig:Ptrick}).
\end{lemma}
\begin{figure}[htbp]
\includegraphics*[width=85mm, bb=
176 550 434 713]{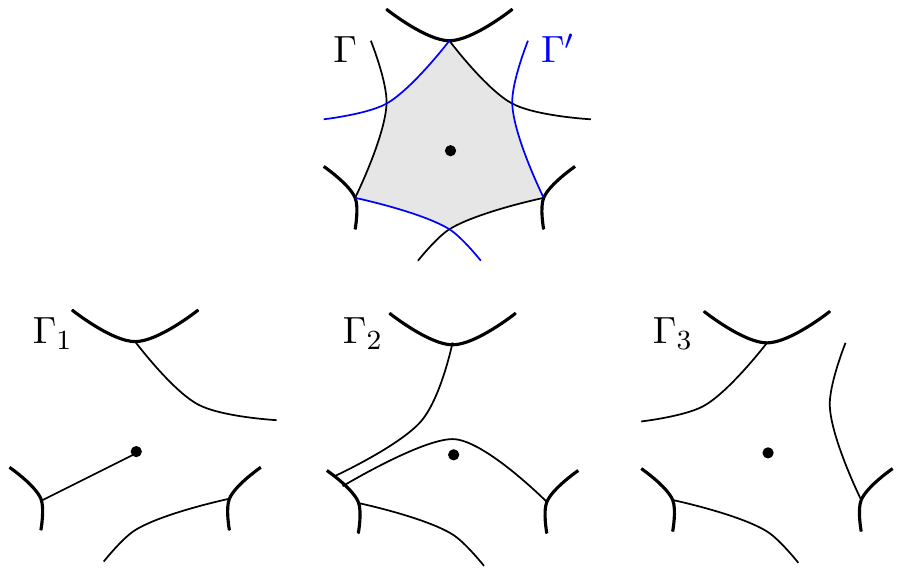}
\caption{Embedded boundary based region with puncture gives $\Gamma \ll_{\ri}^{\partial+P} \Gamma'$: $\Gamma \disjp \Gamma_1 \disjp \Gamma_2 \disjp \Gamma_3=\Gamma'$} 
\label{Fig:Ptrick}
\end{figure}

Let $\pi_S=\pi^{(d)}_{S}: (\widetilde{S}, \widetilde P) \rightarrow (S, P)$ be the standard $d$-fold cyclic branched covering for $d\geq 2$.
For each base point $v_i \in \mathcal{B}$ we choose a lift $\widetilde{v_i} \in \pi^{-1}_S(v_i) \subset \partial \widetilde{S}$ and define a set of base points $\widetilde{\mathcal{B}}=\{\widetilde{v_1},\ldots,\widetilde{v_N}\}$ for $\widetilde{S}$.
For an $N$-arc system $\Gamma=\{\gamma_1,\ldots,\gamma_N\} \in \A$, 
we denote by $\widetilde{\Gamma}=\{\widetilde{\gamma_1},\ldots,\widetilde{\gamma_N}\} \subset \mathcal{A}_{\widetilde{B}}(\widetilde{S},\widetilde{P})$ the $N$-arc system such that $\widetilde{\gamma_i}$ is the lift of $\gamma_i$ with $\widetilde{\gamma_i}(0)=\widetilde{v_i}$.

The following proposition clarifies the relation among coverings and orderings $\ll_{\ri}$ and $\ll_{\ri}^{\partial +P}$.

\begin{proposition}
\label{prop:order-lift}
For $\Gamma,\Gamma' \in \A$, if $\Gamma \ll_{\ri}^{\partial+P} \Gamma'$ in $\AP$ then there exists a standard $d$-fold cyclic branched cover for some $d$ such that 
$\widetilde{\Gamma} \ll_{\ri} \widetilde{\Gamma'}$ in $\mathcal{A}_{\widetilde{B}}(\widetilde{S},\widetilde{P})$. 
\end{proposition}
\begin{proof}
Assume that  $\Gamma \ll_{\ri}^{\partial+P} \Gamma'$. Then there is a sequence of $(\partial +P)$-arc systems $\Gamma_1,\ldots,\Gamma_{k-1}$ such that
\[ \Gamma =:\Gamma_0\disjp \Gamma_1 \disjp \cdots \disjp \Gamma_{k-1} \disjp\Gamma_k:=\Gamma'. \]
For each $\Gamma_i=(\gamma_i^{1},\ldots,\gamma_i^{N}) \in \AP$ we construct $\Gamma_i^{*} = (\gamma^{*}_i {}^{1},\ldots,\gamma^{*}_i{}^{N})\in \A$ as follows: 
For $j=1,\ldots, N$, if the endpoint of the $j$-th arc $\gamma_i^{j}$ lies on $\partial S$, then we define $\gamma^{*}_i {}^{j}:=\gamma_i^{j}$. This means when $\Gamma_i \in \A$ then $\Gamma^{*}_i=\Gamma_i$. If the endpoint of the $j$-th arc $\gamma_i^{j}$ is a puncture point $p \in P$ then we define  
\[ \gamma^{*}_i {}^{j}= \gamma_i^{j} \ast c_p \ast \overline{\gamma_{i}^{j}}. \]
where $\overline{\gamma_{i}^{j}}$ is the path $\gamma_{i}^{j}$ with the reversed orientation, $c_p$ is a small loop around the point $p$ clockwise, and $\ast$ means concatenation of paths. 
We slightly move the end point of $\gamma^{* j}_i$ to the right of $v_i$ along the boundary.

By construction, $\gamma^{*}_{i}{}^{j}$ is disjoint from $\gamma^{*j'}_{i\pm 1}$ whenever $j'\neq j$. When $j'=j$ we have $\gamma^{*j}_i \cap \gamma^{*j}_{i\pm 1} \neq \emptyset$ (see the hollowed point near $v_j$ in Figure~\ref{Fig:lift}). 
However this intersection point can be removed after taking lifts of $\gamma^{*j}_i$ and $\gamma^{*j}_{i\pm 1}$. 
Thus, in the covering space $(\widetilde{S},\widetilde{P})$ we have
\[ \widetilde{\Gamma} \disj \widetilde{\Gamma^{*}_1} \disj \cdots \disj \widetilde{\Gamma^{*}_{k-1}} \disj \widetilde{\Gamma'},
\]
which means $\widetilde{\Gamma} \ll_{\ri} \widetilde{\Gamma'}$  in $\mathcal{A}_{\widetilde{B}}(\widetilde{S},\widetilde{P})$. 
\end{proof}
\begin{figure}[htbp]
\includegraphics*[width=45mm, bb=233 603 377 713]{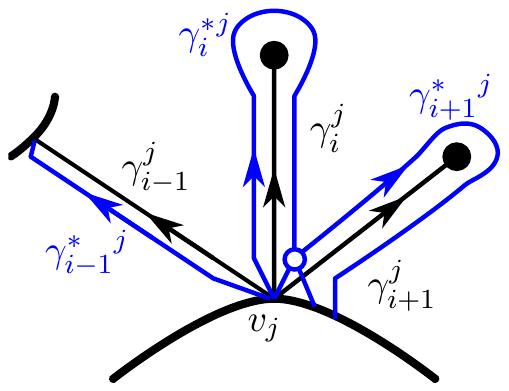}
\caption{Construction of $\gamma^{*j}_i$. The intersection point of $\gamma_i^{*}{}^{j}$ and $\gamma_{i+1}^{*j}$ disappears when we take lifts.} 
\label{Fig:lift}
\end{figure}

\begin{definition}
Let $L$ be a closed braid with respect to an open book $(S, \phi)$.
If there is an $N$-arc system $\Gamma \in \A$ such that 
$$\phi^{tw}_L(\Gamma) \ll^{\partial + P}_{\ri} \Gamma \mbox{ in } \AP$$
and the associated boundary based region $R(\Gamma, \phi_L(\Gamma))$ has type $(N, k)$ then we say that $\phi_L$ is {\em weakly} $(N, k)$-twist left veering. 
\end{definition}

\begin{remark}\label{rem7.6}
When $N=1$, \cite[Corollary 7.6]{IK-qveer} implies that $\gamma \ll^{\partial + P}_{\ri} \gamma'$ if and only if $\gamma \prec_{\ri} \gamma'$. Thus weakly $(1,0)$-twist left-veering is nothing but non-right-veering.
\end{remark}

Theorem~\ref{theorem:non-Nrv-OT} states that if 
$\phi_L$ is $(N, 0)$-twist left veering then $L$ is loose. 
In the next theorem, with a weaker condition we show that $L$ is virtually loose.

\begin{theorem}\label{theorem:virtually-loose}
If $\phi_L$ is weakly $(N, 0)$-twist left-veering 
then $L$ is virtually loose. 
\end{theorem}

\begin{proof} Let $N=1$.
By Remark~\ref{rem7.6}, $\phi_L$ is non-right-veering. 
Then by \cite[Corollary 5.7]{IK-branch}, $L$ is virtually loose.

Let $N>1$. 
Let $\widetilde{\phi_L}:(\widetilde{S}, \widetilde P)\to(\widetilde{S}, \widetilde P)$ be a standard lift of $\phi_L: (S,P)\to(S,P)$ so that with a suitable lift of the base $\mathcal{B}$,  $\widetilde{\phi_L}(\widetilde{\Gamma})$ and $\widetilde{\Gamma}$ can form a non-punctured boundary based region which is a lift of the non-punctured boundary based region $R(\Gamma, \phi_L(\Gamma))$ and 
we have $\widetilde{\phi_L^{tw}(\Gamma)} = \widetilde{\phi_L}{}^{tw}(\widetilde{\Gamma})$.
Since $\phi^{tw}_L(\Gamma) \ll_{\ri}^{\partial + P} \Gamma$  by Proposition~\ref{prop:order-lift} we obtain $\widetilde{\phi_L}{}^{tw}(\widetilde{\Gamma})=\widetilde{\phi_L^{tw}(\Gamma)} \ll_{\ri} \widetilde{\Gamma}$. Thus, $\widetilde{\phi_L}$ is $(N, 0)$-twist left-veering. By Lemma \ref{lemma:virtual-loose} $L$ is virtually loose.
\end{proof}

\end{document}